\documentstyle[amssymb,amsfonts,12pt]{amsart}
\newtheorem{theorem}{Theorem}[section]
\newtheorem{lemma}[theorem]{Lemma}

\newtheorem{proposition}[theorem]{Proposition}
\newtheorem{definition}[theorem]{Definition}

\theoremstyle{remark}
\newtheorem{remark}[theorem]{Remark}

\def\QSet{\mbox{\rm\kern.24em
\vrule width.03em height1.48ex depth-.051ex \kern-.26em Q}}

\def\E{{\mbox{\rm I\kern-.22em E}}}
\def\P{{\bf P}}
\def\D{{\mathcal D}}
\def\T{{\bf T}}
\def\S{{\bf S}}
\def\Z{{\bf Z}}
\def\R{{\bf R}}

\def\O{{\bf O}}
\def\V{{\bf U}}
\def\BMO{{\operatorname{BMO}}}
\def\be#1{\begin{equation}\label{#1}}

\def\size{{\operatorname{size}}}
\def\diam{{\operatorname{diam}}}

\def\H{{\mathcal H}}
\def\Y{{\mathcal R}}
\def\BMO{{\operatorname{BMO}}}
\def\F{{\mathcal F}}

\def\G{{\mathcal G}}

\def\dist{{\operatorname{dist}}}

\def\bas{\begin{align*}}
\def\eas{\end{align*}}
\def\bi{\begin{itemize}}
\def\ei{\end{itemize}}
\newenvironment{proof}{\noindent {\bf Proof} }{\endprf\par}
\def \endprf{\hfill  {\vrule height6pt width6pt depth0pt}\medskip}
\def\emph#1{{\it #1}}

\begin{document}
\title{Multilinear singular operators with fractional rank}

\author{Ciprian Demeter}
\address{Department of Mathematics, Indiana University, 831 East 3rd St., Bloomington IN 47405}
\email{demeterc@@indiana.edu}

\author{Malabika Pramanik}
\address{Department of Mathematics, University of British Columbia, Vancouver, BC CANADA V6T 1Z2}
\email{malabika@@math.ubc.ca}

\author{Christoph Thiele}
\address{Department of Mathematics, UCLA, Los Angeles CA 90095-1555}
\email{thiele@@math.ucla.edu}

\begin{abstract}We prove bounds for multilinear operators on $\R^d$ given by multipliers which are singular along a $k$ dimensional subspace. The new case of interest is when the rank $k/d$ is not an integer. Connections with the concept of {\em true complexity} from Additive Combinatorics are also investigated.
\end{abstract}

\maketitle

\tableofcontents

\section{Introduction}

Let $n\ge 3$ and $d\ge 1$. We consider a multiplier $M(\vec{\xi}^{(1)},\ldots,\vec{\xi}^{(n)})$ on the vector space
$$\Gamma:=\{\vec{\xi}:=(\vec{\xi}^{(1)},\ldots,\vec{\xi}^{(n)})\in (\R^d)^n: \sum_{i=1}^{n}\vec{\xi}^{(i)}=\vec{{\bf 0}}\}.$$
This gives rise to the multi-linear operator on $n-1$ functions on $\R^d$
$$T(F_1,\ldots,F_{n-1})\widehat\;(-\vec{\xi}^{(n)})=\int\delta(\vec{\xi}^{(1)}+\ldots+\vec{\xi}^{(n)})M(\vec{\xi})\widehat{F_1}(\vec{\xi}^{(1)})\cdots\widehat{F_{n-1}}(\vec{\xi}^{(n-1)}) d\vec{\xi}^{(1)}\ldots d\vec{\xi}^{(n-1)}.$$

We will prove the following
\begin{theorem}
\label{ttt1}
Let $\Gamma'$ be a generic linear subspace of $\Gamma$ of dimension $k\ge 0$, and assume
\begin{equation}
\label{reqcc1}
0\le k/d<n/2.
\end{equation}
Assume the multiplier $M:\Gamma\to\R$ satisfies
\begin{equation}
\label{reqcc2}
|\partial^{\alpha}M(\vec{\xi})|\lesssim \dist(\vec{\xi},\Gamma')^{-|\alpha|},
\end{equation}
for all partial derivatives up to some finite order. Then
$$T:L^{p_1}\times\cdots L^{p_{n-1}}\to L^{p_n'},$$
whenever $2<p_i\le \infty$ for each $1\le i\le n$ and
\begin{equation}
\label{homoghjgxjhg}
\frac{1}{p_1}+\ldots+\frac{1}{p_n}=1,
\end{equation}
where $p_n'$ is the conjugate exponent to $p_n$.
\end{theorem}

The generic character of $\Gamma'$ in Theorem \ref{ttt1} is understood with respect to the Lebesgue measure (for example). In fact, we will need $\Gamma'$ to satisfy some precise non-degeneracy conditions, and they are generically satisfied. To give the reader a grasp on what these conditions amount to, we will describe them now in the case $d=2$. The case $d>3$ involves very similar considerations and will be described in Section \ref{secd>2}.

There will be two sets of non-degeneracy requirements. The first one is that $\Gamma'$ can be parameterized by any $k$ of the canonical variables. We will work a lot with the following parametrization.
For $1\le i\le n$, let $G_i:\R^k\to\R^d$ be the linear functions such that
$$G_1(\xi_1,\ldots,\xi_k)=(\xi_1,\ldots,\xi_d)$$
$$G_2(\xi_1,\ldots,\xi_k)=(\xi_{d+1},\ldots,\xi_{2d})$$
$$\ldots\ldots\ldots\ldots$$
$$G_{[k/d]+1}(\xi_1,\ldots,\xi_k)=(\xi_{d[k/d]+1},\ldots,\xi_k,\ldots)$$
where the last entries of $G_{[k/d]+1}$ (the ones after the $\xi_k$ entry) and the entries of the remaining $G_i$ ($[k/d]+1\le i\le n$) are uniquely determined by the requirement that the function
$$G_1\times\ldots\times G_n$$
maps into $\Gamma'$.

Let $m$ be the smallest integer that is greater than or equal to $k/d$.
We will use the notation $\vec{\xi}^{(i)}:=(\xi_1^{(i)},\ldots,\xi_k^{(i)})\in\R^k$. Let $i_1,\ldots,i_k\in\{1,\ldots,n\}$ be pairwise distinct indices. Consider the following system of ($k$ vector, or equivalently $2k$ scalar) linear equations in $2k$ variables $\xi_1^{(1)},\ldots,\xi_k^{(2)}\in\R$, and coefficients $v_j\in\R^2$.
\begin{equation}
\label{solvablesystemchangedputinfront}
\begin{cases} G_{i_1}(\vec{\xi}^{(1)})-G_{i_1}(\vec{\xi}^{(2)})=v_1
\\ G_{i_2}(\vec{\xi}^{(1)})=v_2\\ \ldots \\ G_{i_m}(\vec{\xi}^{(1)})=v_{m}\\ G_{i_{m+1}}(\vec{\xi}^{(2)})=v_{m+1}\\ \ldots \\ G_{i_k}(\vec{\xi}^{(2)})=v_{k}
\end{cases}
\end{equation}

When $d=2$, Theorem \ref{ttt1} has the following precise formulation.

\begin{theorem}
\label{ttt1d-=2}
Let $d=2$ and let $\Gamma'$ be a linear subspace of $\Gamma$ of dimension $k\ge 0$. Assume $\Gamma'$ is the graph over every $k$ of the canonical variables. Moreover, assume that the system \eqref{solvablesystemchangedputinfront} has a unique solution (for each choice of $v_i$) for each pairwise distinct $i_1,\ldots,i_k\in\{1,\ldots,n\}$. If the remaining hypotheses from Theorem \ref{ttt1} are satisfied, then its conclusion will hold.
\end{theorem}

It is not hard to see that the requirements in Theorem \ref{ttt1d-=2} are satisfied for a generic $\Gamma'$. The assumption on the compatibility of the system \eqref{solvablesystemchangedputinfront} is one of many (similar in style) that work for our approach, and are guaranteed to hold generically. Various other possible alternative assumptions will become apparent from our later analysis. The minimal non-degeneracy conditions that are needed for Theorem \ref{ttt1} (or even for our approach) are probably very hard to find, and beyond the goal of this paper.  We point out however that if $\Gamma'$ is degenerate in the sense that it fails to be the graph over {\em some particular} choice of $k$ canonical variables, the analysis of the operator $T$ complicates to a significant extent. This has been observed and investigated in \cite{DT}, in the case $d=2$.

Theorem \ref{ttt1} was proved in \cite{MTT1} in the case $d=1$, so our theorem is only new in the case $d\ge 2$. There, the theorem is proved under just the first non-degeneracy assumption, that $\Gamma'$ is the graph over any $k$ of the canonical variables. The result in \cite{MTT1} is proved for a larger class of indices $p_i$. To simplify our exposition, we choose to prove our theorem in the {\em locally} $L^2$ case $p_i>2$.

A key parameter for our analysis is $m$, introduced earlier. We will refer to $k/d$ as the {\em rank} of the operator. When this rank is an integer (and thus equal to $m$), or more generally, when $m<n/2$, Theorem \ref{ttt1} will follow by a pretty straightforward adaptation of the argument in \cite{MTT1} to the $d-$dimensional setting. The novelty here is that  $k/d$ can be fractional and sufficiently close to $n/2$ to allow for the "bad" case $m\ge n/2$. A first new case of interest where our theorem is applicable is when $d=2$, $k=3$, $n=4$.

A simplified version of our approach also gives an alternative proof to the result in \cite{MTT1} (the $d=1$ case), at least in the case $p_i>2$ (see Section \ref{$d=1$}). Our proof and that from \cite{MTT1} share much of the analytic part of the argument. The proof in \cite{MTT1} however is structured around an induction on $k$ that is not available in the fractional rank case. We eliminate the induction from the argument, and treat all $k$ in a similar fashion. This new type of approach will involve a rather delicate combinatorics.

Theorem \ref{ttt1} also has a kernel formulation:

\begin{theorem}
\label{ttt1xsnxjkjk}
Let $K:\R^{d(n-1)-k}\to \R$ be a Calder\'on-Zygmund kernel, and let $l_i:\R^{d(n-1)-k}\to \R^d$ be $n-1$ generic linear forms. Let also $p_i$ be as in Theorem \ref{ttt1}, and assume $0\le k/d<n/2$. For Schwartz functions $F_1,\ldots,F_n:\R^d\to \R$ define the multilinear operator
$$T(F_1,\ldots,F_{n-1})(\vec{x}):=\int_{\R^{d(n-1)-k}} \prod_{i=1}^{n-1}F_i(\vec{x}+l_i(\vec{t}))K(\vec{t})d\vec{t},\;\;\vec{x}\in \R^d.$$
Then $T$ extends to a bounded operator
$$T:L^{p_1}\times\cdots L^{p_{n-1}}\to L^{p_n'}.$$
\end{theorem}

In the case $k>0$, the operators we investigate will typically have some modulation invariance. As a consequence, proving their boundedness will involve time-frequency analysis similar to the one in the proof of the Bilinear Hilbert Transform (\cite{LT1}, \cite{LT2}).

The assumption $k/d<n/2$ is crucial to our analysis. It can be shown in particular to guarantee that $T$ has no symmetries of higher order (i.e quadratic symmetries). On the other hand, even in the one dimensional case, the quadratic  symmetries may be\footnote{These symmetries are however not {\em guaranteed} to exist. In \cite{GW} there are examples in the case $d=1$, $k=n/2$ which do not have any quadratic or higher order symmetries} present when $k\ge n/2$. Perhaps the most famous example with $k/d=n/2$ is  the Trilinear Hilbert Transform ($n=4$, $k=2$, $d=1$)
\begin{equation}
\label{ejfyewywuiyuieyui1}
T(F_1,F_2,F_3)(x)=\int_{\R}F_1(x+t)F_2(x-t)F_3(x+2t)\frac{dt}{t}.
\end{equation}
Another important example with $k/d=n/2$ is the degenerate two dimensional Bilinear Hilbert Transform ($n=3$, $k=3$, $d=2$)
\begin{equation}
\label{ejfyewywuiyuieyui2}
T(F_1,F_2)(x,y)=\int_{\R}F_1(x+t,y)F_2(x,y+t)\frac{dt}{t},
\end{equation}
where the quadratic symmetries are not singled out, but rather part of an infinite group of symmetries generated by degeneracies (See \cite{DT} for details). The current techniques do not seem enough to address the case $k/d=n/2$, where it is likely that some form of "Quadratic Fourier Analysis" will play a role. A quick single scale heuristics is provided in Section \ref{Single scale heuristics}. However, Theorem \ref{ttt1} above shows that they can address the case of a rank $k/d$ arbitrarily close to (and smaller than) $n/2$. This paper grew partly as an attempt to get more light on these issues.

A second motivation for the considerations in this paper comes from connections with Additive Combinatorics, in particular with the issue of {\em true complexity} of a system of linear equations. Our analysis makes the point that nondegenerate systems characterized by $k/d<n/2$ have true complexity 1, in the language from \cite{GW}. These things are described in Section \ref{Single scale heuristics}. The single scale heuristics provided there sheds a lot of light on the difficulties we encounter in the multi-scale context, and we encourage the reader to go over that section first.

To investigate the boundedness properties of the operator in Theorem \ref{ttt1}, it will be convenient to work with the dualized form defined by
$$\Lambda(F_1,\ldots,F_n):=\int_{\R^d} T(F_1,\ldots,F_{n-1})(\vec{x})F_n(\vec{x})d\vec{x}=$$
$$\int\delta(\vec{\xi}^{(1)}+\ldots+\vec{\xi}^{(n)})M(\vec{\xi})\widehat{F_1}(\vec{\xi}^{(1)})\cdots\widehat{F_{n}}(\vec{\xi}^{(n)}) d\vec{\xi}.$$
We will show that
$$|\Lambda(F_1,\ldots,F_n)|\lesssim \prod_{i=1}^{n}\|F_i\|_{p_i}.$$

In the next section we will discretize the form and convert the problem to the boundedness of a model sum operator. We will use wave packets and multi-dimensional boxes (called tiles) to serve as their Heisenberg boxes. The tiles are then organized into trees, and eventually into certain products of trees, called vector trees. Most of the argument is then devoted to estimating the counting function associated with these vector trees. This is the main new contribution of our paper, and makes the object of Section \ref{sec:counting}.

We will assume $m\ge 2.$
The case $m=0$ (i.e. $k=0$) is entirely classical and goes back to the work of Coifman and Meyer. No modulation symmetries are present in this case.  The case $m=1$ can be addressed by the argument in \cite{MTT1}, by crudely majorizing the rank $k/d$ by $m$. Indeed, since $1=m<3/2\le n/2$, we could treat the operator as if it had rank $m\in\Z$. Alternatively, one can apply the argument from section \ref{$d=1$} here.

We would like to thank Tamara Kucherenko and  Camil Muscalu for helpful discussions on the subject.

The first author acknowledges support by a Sloan Research Fellowship and from NSF grant DMS-0556389.

\section{Discretization}

From now on, the notation $|\cdot|$ will refer to the cardinality of a finite set, the length of an interval or the volume of a multi-dimensional interval, depending on the context. The side length of a cube $R$ will be denoted with $l(R)$. The discretization procedure in this section is very similar to the one from \cite{MTT1}. We omit most of the details.

We will work with the constants
$$1<<C_0<<C_1<<C_2<<C_3<<C_4,$$
whose values will not be specified explicitly, but will rather be clear from the context.
The constant $C_0$ will be chosen first. It will be large enough depending on $\Gamma'$, $n$ and $d$. Then $C_1$ will be chosen large enough, depending on $\Gamma'$ but also on the choice of $C_0$. Then $C_2$, $C_3$ and $C_4$ are chosen in this order, sufficiently large compared to their predecessor in the sequence. $C_4$ will be an integral power of $2.$ No upper bounds will be forced upon $C_i$ in terms of $C_{i-1}$, so when some $C_i$ is selected, it can be chosen as large as desired. The fact that these constants will depend on $\Gamma'$ will be reflected in the fact that the bounds in Theorem  \ref{redu1} below also depend on $\Gamma'$. This dependence will be ignored.

Let $\bar{\O}$ be a finite collection of $nd-$dimensional cubes $\bar{\omega}=\bar{\omega_1}\times\ldots\times\bar{\omega_n}$, where each $\bar{\omega_i}$ is a $d-$dimensional cube. These cubes are a sparse enough subcollection of a Whitney decomposing the frequency space $\R^{nd}\setminus \Gamma'$. They will serve to localize various pieces of the multiplier $M$. For each $1\le i\le n$ define the projections
$$\bar{\O}_i:=\{\bar{\omega_i}:\bar{\omega}\in \bar{\O}\}.$$
These collections will satisfy the following properties:
\begin{itemize}
\item (i) (separation in scale) For each $\bar{\omega_1}\times\ldots\times\bar{\omega_n}\in \bar{\O}$, there is a $l\in\Z$ such that
$$|\bar{\omega_1}|=\ldots=|\bar{\omega_n}|=(C_4)^l$$
\item (ii) (separation in distance) If $\bar{\omega_i}\not=\bar{\omega_i'}\in \bar{\O}_i$ and $|\bar{\omega_i}|=|\bar{\omega_i'}|$ then $\dist(\bar{\omega_i},\bar{\omega_i'})\ge C_4|\bar{\omega_i}|$
\item (iii) (Whitney property) For each $\bar{\omega}\in \bar{\O}$ we have
$$10^{-1}C_0\diam (\bar{\omega})\le \dist(\bar{\omega},\Gamma')\le 10C_0\diam (\bar{\omega})$$
\item (iv) (rank $m$) Any $m$ of the $n$ components of some  $\bar{\omega}\in\bar{\O}$ determine uniquely the remaining $n-m$ components.
\end{itemize}

Let $\D$ be the collection of all dyadic cubes in $\R^d$. Let $\phi$ be a smooth function whose Fourier transform is adapted\footnote{That means supported in $[-1/2,1/2]^d$ and with the first few derivatives bounded by one} to the cube $[-1/2,1/2]^d$. For each $R\in\D$ with center $(c_1(R),\ldots,c_d(R))$ and each $\bar{\omega_i}\in  \bar{\O}_i$ such that $|R||\bar{\omega_i}|=1$ we define the $L^2$ normalized wave-packet
$$\phi_{R\times\bar{\omega_i}}(x_1,\ldots,x_d)=\frac{1}{|R|^{1/2}}\phi\left(\frac{x_1-c_1(R)}{l(R)},\ldots,\frac{x_d-c_d(R)}{l(R)}\right)e^{i(c_1(\bar{\omega_i})x_1+\ldots+c_d(\bar{\omega_i})x_d)}.$$

By using a standard discretization procedure like in \cite{MTT1}, involving a decomposition for the multiplier $M$ adapted to the collection $\bar{\O}$, and then a Gabor basis decomposition for each $F_i$, Theorem \ref{ttt1} will follow from the discretized version below:

\begin{theorem}
\label{redu1}
Let $F_i\in L^{p_i}(\R^d)$ with $2<p_i\le \infty$ as in \eqref{homoghjgxjhg}. Let $\bar{\O}$ be any finite collection satisfying (i)-(iv). Then for generic $\Gamma'$
$$\sum_{R\in\D,\;\bar{\omega}\in  \bar{\O}\atop{|R||\bar{\omega_1}|=1}}|R|^{1-\frac{n}{2}}\prod_{i=1}^{n}|\langle F_i,\phi_{R\times\bar{\omega_i}}\rangle|\lesssim \prod_{i=1}^{n}\|F_i\|_{p_i}.$$
Moreover, the implicit constant in the above inequality is independent of $F_i$ and of the collection $\bar{\O}$.
\end{theorem}

We will fix $\bar{\O}$ for the remaining part of the paper. The following rank properties will be consequences of (i)-(iv) above (see Section 6 in \cite{MTT1} for details) and of the requirement that $\Gamma'$ is the graph over any $k$ of the canonical variables:
\begin{itemize}
\item (v) (overlapping) If for some $\bar{\omega},\bar{\omega}'\in\bar{\O}$ and for some $A\subset\{1,2,\ldots,n\}$ with $|A|=m$ we have
$\bar{\omega_i}\subset 3\bar{\omega_i}'$ for each $i\in A$, then $\bar{\omega_j}\subset C_1\bar{\omega_j}'$ for each $1\le j\le n$
\item (vi) (two lacunary indices) If for some $\bar{\omega},\bar{\omega}'\in\bar{\O}$ with the additional property that $\diam(\bar{\omega})<\diam(\bar{\omega}')$ and for some $A\subset\{1,2,\ldots,n\}$ with $|A|=m$ we have
$\bar{\omega_i}\subseteq 3\bar{\omega_i}'$ for each $i\in A$, then there exist $1\le i_1\not=i_2\le n$ such that $\bar{\omega_j}\nsubseteq 3\bar{\omega_j}'$ for each $j\in\{i_1,i_2\}$.
\end{itemize}

We recall the following definition from \cite{GL1}.
\begin{definition}
A collection $\G$ of intervals in $\R^d$ is called a central grid if
\begin{itemize}
\item ($\G1$) $R,R'\in\G$ and $R\cap R'\not=\emptyset$ implies $R\subseteq R'$ or $R'\subseteq R$
\item ($\G2$) $R,R'\in\G$ and $R\subseteq R'$ implies $C_2R\subseteq R'$
\item ($\G3$) $R,R'\in\G$ and $2|R|<|R'|$ implies $C_2|R|<|R'|$
\item ($\G4$) $R,R'\in\G$ and $|R|<|R'|\le 2|R|$ implies $\dist(R,R')\ge C_2|R'|$
\end{itemize}
\end{definition}

It turns out that each sufficiently sparse collection of cubes can be turned into a central grid. See the considerations following Definition 1 in \cite{GL1} for details.

\begin{lemma}[Centralization]
\label{L:centralization}
Let $\G_0$ be a collection of $d-$dimensional cubes satisfying the following properties
\begin{itemize}
\item (1) $\bar{R},\bar{R'}\in\G_0$ and $|\bar{R}|<|\bar{R'}|$ implies $|\bar{R|}<C_3|\bar{R'}|$
\item (2) $\bar{R},\bar{R'}\in\G_0$ and $|\bar{R}|=|\bar{R'}|$ implies $\dist(\bar{R},\bar{R'})\ge C_3|\bar{R}|$.
\end{itemize}
Then for each $\bar{R}\in \G_0$, there is a $d-$dimensional interval (not necessarily a cube) $R$ such that $\bar{R}\subset R\subset 2\bar{R}$ and such that the collection $\G:=\{R:\bar{R}\in\G_0\}$ is a central grid.
\end{lemma}

Note that both the collection $\bar{\O}_i$ and the collection $2C_1\bar{\O}_i:=\{2C_1\bar{\omega_i}:\bar{\omega_i}\in \bar{\O}_i\}$ satisfy the requirements in Lemma \ref{L:centralization} (if for example $C_3<<C_4-2C_1$). By applying  Lemma \ref{L:centralization} to each of these collections, we associate to each $\bar{\omega_i}\in \bar{\O}_i$ two intervals $\omega_i$ and $\tilde{\omega_i}$ such that $\bar{\omega_i}\subseteq \omega_i\subseteq 2\bar{\omega_i}$, $2C_1\bar{\omega_i}\subseteq \tilde{\omega_i}\subseteq 4C_1\bar{\omega_i}$, and such that the collection
$$\O_i:=\{\omega_i:\bar{\omega_i}\in \bar{\O}_i\}$$
is a grid, while the collection
$$\tilde{\O_i}:=\{\tilde{\omega_i}:\bar{\omega_i}\in \bar{\O}_i\}$$
 is a central grid. Note that we will need a stronger assumption on the second collection. For each $\bar{\omega}=\bar{\omega_1}\times\ldots\times\bar{\omega_n}$, we will use the notation ${\omega}={\omega_1}\times\ldots\times{\omega_n}$ and $\tilde{\omega}=\tilde{\omega_1}\times\ldots\times\tilde{\omega_n}$, and these two new intervals will form the collections $\O$ and $\tilde{\O}$, respectively.

From now on we will abandon the collection $\bar{\O}$ and only refer to the collections $\O$ and $\tilde{\O}$.

Note that the sidelegths of each interval ${\omega}\in\O$ are within a factor of two from each other. Denote by $\Y({\omega_i})$ the collection of all dyadic cubes $R$ in $\R^d$ such that $|R||\bar{\omega_i}|=1$. Since $\bar{\omega}$ is a cube, the collections $\Y({\omega_i})$ are all the same for $1\le i\le n$. For each $1\le i\le n$, let $\S_i$ denote the collection of all  $s_i:=R_{s_i}\times \omega_{s_i}$, with $\omega_i$ ranging through $\O_i$ and $R_{s_i}$ ranging through $\Y({\omega_i})$.
The collection $\S$ will consist of all $s=R_s\times \omega_{s}$, with $\omega_{s}:=\omega_{s_1}\times\ldots\omega_{s_n}$ ranging through $\O$ and $R_s$ ranging through (any of the) $\Y({\omega_{s_i}})$.

An element $s=R_s\times \omega_{s_1}\times\ldots\omega_{s_n}\in \S$ will be referred to as a multi-tile, while its components $s_i:=R_s\times\omega_{s_i}\in \S_i$ will be referred to as tiles. The intervals $R_s$ and $\omega_s$ will be referred to as the spatial and frequency components of $s$.

We introduce some relations of order, which are very similar to the ones in Definition 6.1. in \cite{MTT1}.

\begin{definition}
Let $s_i,s_i'\in\S_i$ be two tiles. We write
\begin{itemize}
\item $s_i\le s_i'$ if $R_{s_i}\subseteq R_{s_i'}$ and $\omega_{s_i'}\subseteq \omega_{s_i}$
\item $s_i\lesssim s_i'$ if $R_{s_i}\subseteq R_{s_i'}$ and $\tilde{\omega
}_{s_i'}\subseteq \tilde{\omega}_{s_i}$
\item $s_i\lesssim' s_i'$ if $s_i\lesssim s_i'$ but $s_i\nleq s_i'$
\end{itemize}
\end{definition}

We note that as a consequence of the rank properties (iv)-(vi) and the grid structure of $\O$ and $\tilde{\O}$ we get the following
\begin{itemize}
\item (r1) (rank $m$) Any $m$ of the $n$ frequency components of some  $s\in\S$ determine uniquely the remaining $n-m$ components.
\item (r2) (overlapping) If for some $s,s'\in\S$ and for some $A\subset\{1,2,\ldots,n\}$ with $|A|=m$ we have
$s_i\le s_i'$ for each $i\in A$, then $s_j\lesssim s_j'$ for each $1\le j\le n$
\item (r3) (two lacunary indices)
If for some $s,s'\in\S$ with $|I_{s'}|<|I_s|$ and for some $A\subset\{1,2,\ldots,n\}$ with $|A|=m$ we have
$s_i\le s_i'$ for each $i\in A$, then there exist $1\le i_1\not=i_2\le n$ such that $s_{j}\lesssim's_j'$ for each $j\in\{i_1,i_2\}$.
\end{itemize}

Properties (r2) and (r3) are all about the frequency components of multi-tiles, the spatial components do not play any role.

We also record for future reference the grid properties satisfied by the multi-tiles:
\begin{itemize}
\item (r4) The collection $\{\omega_{s_i}:s_i\in\S_i\}$ is a grid\footnote{While we can arrange that this collection is a {\em central} grid, too, we will not need this strong assumption}
\item (r5) The collection $\{\tilde{\omega}_{s_i}:s_i\in\S_i\}$ is a central grid
\end{itemize}
It is also clear (due to (iii)) that if $C_0$ is sufficiently large, then
\begin{itemize}
\item (r6) For each $s\in \S$ we have $C_0^2{\omega}_s\cap \Gamma'\not=\emptyset$
\end{itemize}
For each tile $s_i=R_{s_i}\times \omega_{s_i}$, we will have a wave packet associated with it, namely
$$\phi_{s_i}:=\phi_{R_{s_i}\times \bar{\omega}_{s_i}},$$
where $\bar{\omega}_{s_i}$ is the cube in $\bar{\O_i}$ that generates $\omega_{s_i}$ via the procedure described earlier. Note that the Fourier transform of $\phi_{s_i}$ is supported in $\omega_{s_i}$, while spatially, $\phi_{s_i}$ is a bump function quasi-localized near the cube $R_{s_i}$.

With this notation, Theorem \ref{redu1} can be rephrased as follows:

\begin{theorem}
\label{redu2}
Let $F_i\in L^{p_i}(\R^d)$ with $2<p_i\le \infty$ as in \eqref{homoghjgxjhg}. Let $\S$ be a collection of multi-tiles satisfying (r1)-(r6). Then for generic $\Gamma'$
$$\sum_{s\in\S}|R_s|^{1-\frac{n}{2}}\prod_{i=1}^{n}|\langle F_i,\phi_{s_i}\rangle|\lesssim \prod_{i=1}^{n}\|F_i\|_{p_i}.$$
Moreover, the implicit constant in the above inequality is independent of $F_i$ and of the collection $\S$.
\end{theorem}

The genericity of $\Gamma'$ will imply an additional rank property (r7), that we choose not to state explicitly at this point, but which will become rather clear later (for example, see the beginning of Section \ref{The case $d=2$}).

\section{Trees}
In order to prove Theorem \ref{redu2} we will organize each collection $\S_i$ into smaller structures, called trees.

\begin{definition}
Let $R_{\T}$ be a dyadic cube in $\R^d$ and let $\xi_{\T}\in\R^d$. A collection of tiles $\T\subset \S_i$ is called an $i-$tree (also sometimes referred to as tree, when the index $i$ is either not important or when it is clear from the context) with top $(R_{\T},\xi_{\T})$, if $R_{s_i}\subseteq R_\T$ and  $\xi_{\T}\in \tilde{\omega}_{s_i}$ for each  $s_i\in\T$.

In case $\xi_{\T}\in \omega_{s_i}$ for each $s_i\in\T$ the tree $\T$ is called $i-$overlapping. If $\xi_{\T}\in \tilde{\omega}_{s_i}\setminus\omega_{s_i}$ for each $s_i\in\T$, the tree  will be called $i-$lacunary.
\end{definition}

We note that a tree consisting of a single tile is both lacunary and overlapping (these are actually the only examples of such trees). In general, a tree must not necessarily be overlapping or lacunary. However, each tree can be split as the disjoint union of an overlapping tree and a lacunary tree.

Trees will be used to construct similar structures consisting of multi-tiles, called vector trees.

\begin{definition}
Let $R_{\vec{\T}}$ be a dyadic cube in $\R^d$ and let $\xi_{\vec{\T}}=(\xi_{\vec{\T},1},\dots,\xi_{\vec{\T},n})\in\R^{dn}$. A collection of multi-tiles $\vec{\T}\subset \S$ is called a vector tree with top  $(R_{\vec{\T}},\xi_{\vec{\T}})$ if for each $1\le i\le n$, the projection $\T_i:=\{s_i:s\in\vec{\T}\}$ is an $i-$tree with top $(R_{\vec{\T}},\xi_{\vec{\T},i})$.

The vector tree $\vec{\T}$ is called $i-$overlapping, if its projection $\T_i$ is an $i-$overlapping tree. If this is the case, we call the index $i$ overlapping. Similarly, the vector tree $\vec{\T}$ is called $i-$lacunary, if $\T_i$ is an $i-$lacunary  tree. If this is the case, we call the index $i$ lacunary.
\end{definition}

\begin{remark}
Note that the rank property (r3) implies that each vector tree has at least two lacunary indices.
\end{remark}

\begin{definition}
Let $\P\subseteq\S_i$ be a collection of tiles. Its $i-$size is defined as
$$\size_i(\P):=\sup_{\T\subseteq\P}\left(\frac{1}{|R_{\T}|}\sum_{s_i\in\T}|\langle F_i,\phi_{s_i}\rangle|^2\right)^{1/2},$$
where the supremum is taken over all lacunary $i-$trees $\T\subseteq\P$ with tops $(R_{\T},\xi_{\T})$.
\end{definition}

The definition of the size of a collection $\P$ depends on the choice function $F_i$. We choose not to index the size by $F_i$, since $F_i$ will later be fixed.

The next Lemma shows that the size is dominated by the Hardy-Littlewood maximal function (see again \cite{MTT1} for details)
$$M_2F(\vec{x}):=\sup_{\vec{x}\in R\atop{R\;cube\;in\;\R^d}}\left(\frac{1}{|R|}\int_{R}|F|^2(\vec{y})d\vec{y}\right)^{1/2}.$$

\begin{lemma}
\label{sizeestbymaxf}
Let $\T$ be a lacunary $i-$tree with top $(R_{\T},\xi_{{\T}})$. Then
$$\left(\frac{1}{|R_{\T}|}\sum_{s_i\in\T}|\langle F_i,\phi_{s_i}\rangle|^2\right)^{1/2}\lesssim \inf_{\vec{x}\in R_\T}M_2F_i(\vec{x}).$$
As a consequence,
$$\size_i(\P)\lesssim \sup_{s_i\in\P}\inf_{\vec{x}\in R_{s_i}}M_2F_i(\vec{x}).$$
\end{lemma}

We will use the following estimate for vector trees:
\begin{lemma}
\label{vecttreeest}
Let $\vec{\T}$ be a vector tree with top $(R_{\T},\xi_{\T})$. Then
$$\sum_{s\in\T}|R_s|^{1-\frac{n}{2}}\prod_{i=1}^{n}|\langle F_i,\phi_{s_i}\rangle|\le |R_{\T}|\prod_{i=1}^{n}\size_i(\T_i).$$
\end{lemma}
\begin{proof}
Use H\"older's inequality with an $l^2$ estimate for two lacunary indices and with an $l^{\infty}$ estimate for the remaining $n-2$ indices.
\end{proof}

For a collection $\F$ of trees or vector trees, we will use the notation
$$N_{\F}(\vec{x}):=\sum_{\T\in\F}1_{R_\T}(\vec{x}).$$
Also, we will denote with $\|N_{\F}\|_{\BMO}$ the dyadic BMO norm of $N_{\F}$.

We next show how to split a collection $\P$ of tiles into collections of trees with controlled size. First, we show how to cut the size in half.

\begin{lemma}
\label{l:Bes1}
Let $\P\subseteq \S_i$ be a collection of tiles. There is a collection $\F$ of disjoint (as collections of tiles) $i-$trees in $\P$ such that
\begin{equation}
\label{ghxcweytdutdleu gnd befvuy45ufyigsdghgfuetyferym,.m,.m}
R_\T\subset \{\vec{x}:\inf_{\vec{x}\in R_{\T}}M_2F_i(\vec{x})\gtrsim \size_i(\P)\},\;\;\T\in\F
\end{equation}
\begin{equation}
\label{ghxcweytdutdleu gnd befvuy45ufyigsdghgfuetyfery}
\sum_{\T\in\F}|R_{\T}|\lesssim (\size_i(\P))^{-2}\|F_i\|_2^2
\end{equation}
\begin{equation}
\label{ghxcweytdutdleu gnd befvuy45ufyigsdghgfuetyferyghcgjhkla}
\|N_{\F}\|_{\BMO}\lesssim (\size_i(\P))^{-2}[\sup_{\T\in\F}\inf_{\vec{x}\in R_{\T}}M_2F_i(\vec{x})]^2
\end{equation}
\begin{equation}
\label{ghxcweytdutdleu gnd befvuy45ufyigsdghgfuetyferyghcgjhkladjhfjkdhsaAsa}
\|N_{\F}\|_{q}\lesssim (\size_i(\P))^{-2}[\sup_{\T\in\F}\inf_{\vec{x}\in R_{\T}}M_2F_i(\vec{x})]^2[\size_i(\P)^{p_i}\|F_i\|_{p_i}^{p_i}]^{1/q},\;1\le q<\infty
\end{equation}
and
$$\size_i(\P\setminus\bigcup_{\T\in\F}\T)<\frac{1}{2}\size_i(\P).$$
\end{lemma}
\begin{proof}
The proof is very standard, so we only say a few words about it (see for example Lemma 7.7. in \cite{MTT1} or Lemma 6.6. in \cite{MTT3} for details). The trees in $\F$ are selected in $2^d$ stages. We need first a definition. Let $1\le l\le d$. A lacunary $i-$tree with top $(R_{\T},\xi_{\T})$ is said to be a $(+,l)$ tree if
\begin{equation}
\label{ghxcweytdutdleu gnd befvuy45ufyi}
(\xi_{\T})_l>c_l(\omega_{s_i})
\end{equation}
for each $s_i\in\T$, where $c_l(\omega)$ denotes the $l^{th}$ component of the center $c(\omega)$ of $\omega$. Similarly, the tree is said to be a $(-,l)$ tree if the inequality is reversed in \eqref{ghxcweytdutdleu gnd befvuy45ufyi}. It is easy to see that each lacunary tree $\T$ is the disjoint union of at most $2^d$ lacunary subtrees having the same top as $\T$, each of which is either a $(+,l)$ tree or a $(-,l)$ tree for some $1\le l\le d$. By pigeonholing, it follows that if we eliminate from $\P$ all such trees having $\size_i$ at least $\frac{1}{2\times 2^{d/2}}\size_i(\P)$, the remaining collection of tiles will have the size $<\frac{1}{2}\size_i(\P)$.

 In a typical stage of the selection\footnote{The order of the stages does not really matter} one selects lacunary trees $\T$ with tops $(R_{\T},\xi_{\T})$ which are (say) $(+,l)$ trees, for a fixed $l$, and which satisfy
\begin{equation}
\label{jhfgfgtreeelim}
\frac{1}{|R_{\T}|}\sum_{s_i\in\T}|\langle F_i,\phi_{s_i}\rangle|^2\ge \frac{1}{4\times 2^{d}}(\size_i(\P))^2.
\end{equation}
 One always aims for the tree which has the minimal value for $(\xi_{\T})_l$ among all the trees that qualify to be selected at that moment. After such a tree  is selected, this tree is added to the collection $\F_1(+,l)$ and its tiles are eliminated from $\P$. Note that the remaining tiles which satisfy $\xi_{\T}\in\tilde{\omega}_{s_i}$ and $R_{s_i}\subseteq R_{\T}$ will form a tree $\T^{sat}$, having the same top as $\T$. Add $\T^{sat}$ to the collection $\F_2(+,l)$ and eliminate its tiles from $\P$. Then the cycle repeats, that is one starts searching again for a $(+,l)$ tree satisfying  \eqref{jhfgfgtreeelim}. When no such tree is available, the  $(+,l)$ stage of the selection process ends. One adds all the trees from  $\F_1(+,l)$ and from $\F_2(+,l)$  to the collection $\F$. At this point one goes to the next stage of the selection process.

The crucial observation is that in each stage, the trees from the collection $\F_1(+,l)$ have the following property\footnote{In the literature, this property is referred to as "strong disjointness"}: If $\T,\T'\in \F_1(+,l)$, $s_i\in\T$, $s_i'\in\T'$ and $\omega_{s_i}\subsetneq\omega_{s_i'}$, then $R_{s_i'}\cap R_{\T}=\emptyset$.

Let us briefly see why this property holds. The fact that $\omega_{s_i}\subsetneq\omega_{s_i'}$, the separation in scales, (r4) and the definition of $(+,l)$ lacunary trees implies that $(\xi_{\T'})_l>(\xi_{\T})_l$. This in turn implies that the tree $\T$ was selected earlier than $\T'$. The separation in scales and the fact that $C_1<<C_2$ easily imply that $\xi_{\T}\in \tilde{\omega}_{s_i'}$. If $R_{s_i'}$ and $R_{\T}$ intersected (and this can only mean that $R_{s_i'}\subseteq R_{\T}$), this together with  $\xi_{\T} \in \tilde{\omega}_{s_i'}$ would imply  that right after $\T$ was selected, $s_i'$ would have qualified to be part of the tree $\T^{sat}$, and thus it would have been eliminated from $\P$ before the selection of $\T'$ began. The contradiction is immediate, thus we conclude that $R_{s_i'}\cap R_{\T}=\emptyset$.

What the property we just proved means, is that any two trees from $\F_1(+,l)$
 interact very weakly, in that tiles from different trees either have disjoint frequency components or strongly separated spatial components (so that the associated wave packets have little interaction). Using standard arguments, one could argue that this together with \eqref{jhfgfgtreeelim} implies the following Bessel type inequality
$$\sum_{\T\in\F_1(+,l)}|R_{\T}|\lesssim (\size_i(\P))^{-2}\|F_i\|_2^2.$$
We can clearly replace $\F_1(+,l)$ by $\F_2(+,l)$ in the above inequality.

By combining the contribution from all $2^d$ stages we get \eqref{ghxcweytdutdleu gnd befvuy45ufyigsdghgfuetyfery}. Also, \eqref{ghxcweytdutdleu gnd befvuy45ufyigsdghgfuetyferyghcgjhkla} will follow similarly, by a standard localization argument.
\eqref{ghxcweytdutdleu gnd befvuy45ufyigsdghgfuetyferyghcgjhkladjhfjkdhsaAsa} follows from \eqref{ghxcweytdutdleu gnd befvuy45ufyigsdghgfuetyferym,.m,.m} and \eqref{ghxcweytdutdleu gnd befvuy45ufyigsdghgfuetyferyghcgjhkla}, via an application of John-Nirenberg's inequality (since $p_i>2$). Finally, \eqref{ghxcweytdutdleu gnd befvuy45ufyigsdghgfuetyferym,.m,.m} is a consequence of Lemma \ref{sizeestbymaxf}.
\end{proof}

By iterating Lemma \ref{l:Bes1} we get
\begin{proposition}
\label{p:Bes2}
Let $\P_i\subseteq \S_i$ be a collection of tiles. Then one has the following decomposition
$$\P_i:=\bigcup_{2^{-k}\le \size_i(\P_i)} \P^{(k)}_i,$$
where
$$\size_i(\P^{(k)}_i)\le 2^{-k+1}$$
and each $\P^{(k)}_i$ is the (disjoint) union of a family $\F^{(k)}_i$ of trees such that
$$\|N_{\F^{(k)}_i}\|_{q}\lesssim 2^{2k}[\sup_{\T\in\F^{(k)}_i}\inf_{\vec{x}\in R_{\T}}M_2F_i(\vec{x})]^2[2^{k_ip_i}\|F_i\|_{p_i}^{p_i}]^{1/q},\;1\le q<\infty.$$
\end{proposition}

This proposition gives us good control over the number of trees corresponding to each component $i\in\{1,\ldots,n\}$. In the next section we will combine trees from each component to create vector trees, and will relate their counting function to the ones of the individual trees.

\section{Counting trees}
\label{sec:counting}
Recall that $m$ is the smallest integer greater than or equal to $k/d$, and that we have assumed that $m\ge 2$.

Throughout this section we will assume that we have a collection of multi-tiles $\P\subset\S$. We also assume that each projection $\P_i$ consists of a family $\F_i$ of disjoint $i-$trees
$$\P_i=\bigcup_{\T\in\F_i}\T.$$

Our goal is to split the collection $\P$ into a family $\F$ of vector trees with good pointwise control over the counting function $N_\F(\vec{x})$ of their tops in terms of the individual counting functions $N_{\F_i}(\vec{x})$ (see \eqref{estforcountfunction1fgfdgfh} and \eqref{chvhcggyeuiwqyduy3e8394945056=-067asxwqr3466uuk} below). To provide the reader with a better understanding of what we have to prove, we first give a single scale heuristics.

\subsection{Single scale heuristics}
\label{Single scale heuristics}

Assume we are in the particular case where each $\T$ in each family $\F_i$ consists of just one tile, of the form $[0,1]^{d}\times \omega$. Thus, the collection $\F_i$ will consist of a family of pairwise disjoint tiles of scale 1. Assume that for each $i$ we know the cardinality $|\F_i|$. The question in this case is, subject to axioms (r1)-(r6) and the genericity of $\Gamma'$, to estimate the cardinality $|\F|$ where $\F$ consists of all the multi-tiles $s$ having the property that $s_i\in \F_i$ for each $i$. Note first that (r1) immediately implies that
\begin{equation}
\label{hdcwery389`90ujdfbvhdfe3ydecfe}
|\F|\le \prod_{i=1}^{m}|\F_i|,
\end{equation}
with the same being true for each choice of $m$ indices. This leads to
$$|\F|\le \prod_{i=1}^{n}|\F_i|^{m/n}.$$
This estimate is only satisfactory when $m<n/2$, for reasons that will become clear in the proof of Proposition \ref{propproppropprop} below, however, it will be of no use when $m\ge n/2$. The good news is that \eqref{hdcwery389`90ujdfbvhdfe3ydecfe} is only sharp when $k/d=m$, in which case we also know that $m<n/2$. If $k/d<m$, the inequality can be improved, and one has to use the additional rank property (r7) guaranteed by the genericity of $\Gamma'$.

For simplicity, let us see this in the particular case $n=4,$ $d=2$, $k=3$. The additional rank property (r7) in this case will state that for each pairwise distinct $i_1,i_2,i_3\in\{1,2,3,4\}$, and for every two multi-tiles $s,s'$ with $s_1=s_1'$, knowledge of both $s_2$ and $s_3'$ will uniquely determine both $s$ and $s'$. See the beginning of Section \ref{The case $d=2$} for details. Using this and applying  Lemma \ref{lemakatztao} as in the next section, we get the improved inequality
$$|\F|\le \prod_{i=1}^{3}|\F_i|^{1/2},$$
and thus, after permuting indices, we get
$$|\F|\le \prod_{i=1}^{4}|\F_i|^{3/8}.$$
In general, one gets
\begin{equation}
 \label{wq7t76e79r=-4omvm nhuirfgrf421984=32=0-}
|\F|\lesssim \prod_{i=1}^{n}|\F_i|^{\frac{k/d}{n}},
\end{equation}
which is the optimal inequality. The important fact is that all exponents on the right hand side are $<1/2$.

As a consequence of \eqref{wq7t76e79r=-4omvm nhuirfgrf421984=32=0-} we observe
\begin{proposition}
\label{propproppropprop}
Assume as before that $k/d<n/2$. Then, for generic linear forms $l_i:\R^{d(n-1)-k}\to\R^d$, $1\le i\le n-1$ we have
$$|\int_{\vec{x}\in [0,1]^{d}}\int_{\vec{t}\in [0,1]^{d(n-1)-k}} \prod_{i=1}^{n-1}F_i(\vec{x}\oplus l_i(\vec{t}))F_n(\vec{x})d\vec{t}d\vec{x}|\lesssim \prod_{i=1}^{n}\|\hat{F_i}\|_{\infty}^{1-\frac{2k/d}{n}},$$
whenever $\|F_i\|_{\infty}=O(1)$.
\end{proposition}
Here $\oplus$ denotes addition modulo 1.
This shows that the single scale operator relevant to our problem is controlled by the Fourier transform, and thus, it provides a heuristics for the fact that Fourier analysis will be able to address the multi-scale version of the problem (i.e. Theorem \ref{redu2}). In contrast to this, we mention that neither the operator in \eqref{ejfyewywuiyuieyui1} nor that in \eqref{ejfyewywuiyuieyui2} satisfy any similar bounds.

In the language of Gowers-Wolf from \cite{GW}, Proposition \ref{propproppropprop} is saying that the system of equations associated with the linear forms $l_i$ has {\em true complexity} 1.

It will become clear from the argument presented in the following sections that the precise conditions on $\Gamma'$ needed in Proposition \ref{propproppropprop} (that is, needed to guarantee \eqref{wq7t76e79r=-4omvm nhuirfgrf421984=32=0-}) amount to the following two requirements. To formulate them, we use the same notation that we have used so far. More exactly, $\Gamma'$ is the linear subspace of $(\R^d)^n$ defined by
$$\{(\vec{\eta}^{(1)},\ldots,\vec{\eta}^{(n)}):\sum_{i=1}^{n-1}\vec{\eta}^{(i)}\cdot[\vec{x}+l_i(\vec{t})]+\eta^{(n)}\cdot\vec{x}={\bf 0}\in \R[\vec{x},\vec{t}]\}.$$
We ask $\Gamma'$ to be $k$ dimensional and parametrizable over every $k$ canonical coordinates. We also ask that the following system of $k$ $\R^d$ valued equations in $d$ unknowns $\vec{\xi}^{(i)}\in \R^k$ is compatible for each distinct $i_1,\ldots, i_m\in\{1,\ldots,n\}$
$$
\begin{cases} G_{i_1}(\vec{\xi}^{(j)})-G_{i_1}(\vec{\xi}^{(j+1)})=v_{j},\;\;1\le j\le d-1
\\ G_{i_l}(\vec{\xi}^{(j)})=v_{j,i_l},\;\;1\le j\le d,\;2\le l\le m-1\\ G_{i_m}(\vec{\xi}^{(j)})=v_{j,m},\;\;1\le j\le k+1-d(m-1).
\end{cases}
$$

\begin{proof}(of Proposition \ref{propproppropprop})
By discretizing as in the previous sections (and keeping the notation from there), we are reduced to proving that
$$\sum_{s\in\S}\prod_{i=1}^{n}|\langle F_i,\phi_{s_i}\rangle|\lesssim \prod_{i=1}^{n}\|\hat{F_i}\|_{\infty}^{1-\frac{2k/d}{n}},$$
where $\S$ consists of multi-tiles of scale 1. Note that $|\langle F_i,\phi_{s_i}\rangle|$ is roughly a Fourier coefficient of $F_i$.
For each $2^{-n_i}\lesssim \|\hat{F_i}\|_{\infty}$ let $\F_{i,n_i}$ be the set of all tiles $s_i$ with
$$|\langle F_i,\phi_{s_i}\rangle|\sim 2^{-n_i}.$$
Note that since $\|F_i\|_2=O(1),$ we have that $|\F_{i,n_i}|\lesssim 2^{2n_i}$ (Bessel's inequality). If we use \eqref{wq7t76e79r=-4omvm nhuirfgrf421984=32=0-}, it follows that the cardinality of the set $\F(\vec{n})$ consisting of all multi-tiles $s\in\S$ such that $s_i\in \F_{i,n_i}$ for each $i$, will be $O(\prod_{i=1}^{n}2^{\frac{2kn_i/d}{n}})$.

The  sum above is then bounded by
$$\sum_{\vec{n}:2^{-n_i}\lesssim \|\hat{F_i}\|_{\infty}}
2^{-\sum_i n_i}\prod_{i=1}^{n}2^{\frac{2kn_i/d}{n}}\lesssim \prod_{i=1}^{n}\|\hat{F_i}\|_{\infty}^{1-\frac{2k/d}{n}}.$$
\end{proof}

We make two remarks about Proposition \ref{propproppropprop}, and in general about the applicability of our Fourier analytic techniques. The first remark shows that some non-degeneracy is in general needed. Take for example
$$\int_{(t,s)\in\T^2}\int_{(x,y)\in\T^2}F_1(x+t,y)F_2(x,y+s)F_3(x,y)dtdsdxdy$$ $$=\int_{(x,y)\in\T^2}\F_x(F_1)(0,y)\F_y(F_2)(x,0)F_3(x,y)dxdy.$$
One can check that here $k=d=2$, and thus $k/d<n/2$.
However, it can be easily seen that this form can not be bounded by a power of (and in general, by no reasonable function of) $\|\widehat{F_i}\|_{\infty}$. The explanation is that $\Gamma'$ is degenerate.

The second remark will show that with the current techniques, the requirement $k/d<n/2$ can not be relaxed. We illustrate this in the case $d=1$ with the single scale version of \eqref{ejfyewywuiyuieyui1}
$$\int_{x\in [0,1]}\int_{t\in [0,1]}F_1(x\oplus t)F_2(x\ominus t)F_3(x\oplus 2t)F_4(x)dxdt.$$

Application of the Fourier inversion formula for each function shows that the form above is essentially (up to some constants)
$$\sum_{(n_1,n_2,n_3,n_4)\atop{n_1-n_2+2n_3=0\atop{n_1+n_2+n_3+n_4=0}}}\prod_{i=1}^4\widehat{F_i}(n_i).$$
Our approach relies on bounding the sum above by using the triangle inequality, by
\begin{equation}
\label{triangle7492056fg}
\sum_{(n_1,n_2,n_3,n_4)\atop{n_1-n_2+2n_3=0\atop{n_1+n_2+n_3+n_4=0}}}\prod_{i=1}^4|\widehat{F_i}(n_i)|.
\end{equation}

Choose $F_1(x)=F_2(x)=F_3(x)=e^{2\pi iNx^2}$ and $F_4(x)=e^{-2\pi iNx^2}$, for large $N$. One can check that when $i\in\{1,2,3\}$,  $|\widehat{F_i}(n)|\sim N^{-1/2}$ for $N/100\le n\le N$, and that
$|\widehat{F_4}(n)|\sim N^{-1/2}$ for $-N\le n\le -N/100$. It is easy to see that the term in \eqref{triangle7492056fg} is $\gtrsim 1$. Since one can check that also $\|\widehat{F_i}\|_{\infty}\lesssim N^{-1/2}$ for each $i$, the expression in \eqref{triangle7492056fg} is not $O(\|\widehat{F_i}\|_{\infty})$. Let $\F_i$ refer to the collection of $n$ with $|\widehat{F_i}(n)|\sim N^{-1/2}$, and let $\F$ be the collection of quadruples $(n_1,n_2,n_3,n_4)$ as in \eqref{triangle7492056fg}, such that $n_i\in\F_i$ for each $i$. We get that
$$|\F|\sim\prod_{i=1}^{4}|\F_i|^{1/2},$$
and this sharp inequality becomes inefficient for any application, due to the exponents being $1/2$.

A moment's reflection shows that we can get the same outcome whenever $k=n/2$. Of course, our approach fails to address the case $k/d=n/2$, because of the use of triangle's inequality in \eqref{triangle7492056fg}. It is likely that the correct approach in this case (and in general, whenever $k/d\ge n/2$) is by analyzing appropriate quadratic Fourier coefficients.

To address the multi-scale case, we will have to count vector trees, rather than just multi-tiles. To make the exposition more transparent, we will  exemplify our approach in the case $d=2$. The general case $d>2$ will follow via a similar argument, and is briefly sketched afterwards.

\subsection{The case $d=2$}
\label{The case $d=2$}
We start this section by proving that, under the particular assumptions on $\Gamma'$ from Theorem \ref{ttt1d-=2}, the multi-tiles in $\S$ satisfy the following additional rank property:
\begin{itemize}
\item (r7) For every distinct indices $i_1,\ldots,i_k\subset\{1,2,\ldots,n\}$ and for every (not necessarily pairwise distinct) multi-tiles $s,s',p,p'\in\S$ we have
$$\dist(\omega_s,\omega_p), \dist(\omega_{s'},\omega_{p'})\le C_0^5D_{max},$$
where
$$D_{max}:=\max\{\dist(\omega_{s_{i_1}},\omega_{s_{i_1}'}), \dist(\omega_{p_{i_1}}, \omega_{p_{i_1}'})\}+$$
$$\max\{\dist(\omega_{s_{i_l}}, \omega_{p_{i_l}}),2\le l\le m\}+\max\{\dist(\omega_{s_{i_l}'}, \omega_{p_{i_l}'}),m+1\le l\le k\}+$$
\begin{equation}
\label{t46e7wncgd}
+\max\{\diam(\omega_s),\diam(\omega_{s'}),\diam(\omega_p),\diam(\omega_{p'})\}.
\end{equation}
\end{itemize}

It will be important that we can control both $\dist(\omega_s,\omega_p)$ and $\dist(\omega_{s'},\omega_{p'})$ by a value at most $C_0^5$ times larger than $D_{max}$. (r7) will be used twice in the future, in conjunction with a choice of constants such that $C_0<<C_1<<C_2$.

To see why (r7) holds, let $\gamma_s\in C_0^2\omega_s\cap \Gamma'$, $\gamma_{s'}\in C_0^2\omega_{s'}\cap \Gamma'$,  $\gamma_p\in C_0^2\omega_p\cap \Gamma'$ and $\gamma_{p'}\in C_0^2\omega_{p'}\cap \Gamma'$ (by (r6)). To simplify notation, assume $i_l=l$. Note that

\begin{equation}
\label{solvablesystemsahhdty823r3409rftk obrjhituhi}
\dist((\gamma_{s})_{1},(\gamma_{s'})_{1}), \dist((\gamma_{p})_{1}, (\gamma_{p'})_{1}), \dist((\gamma_{s})_{l}, (\gamma_{p})_{l}), \dist((\gamma_{s'})_{l}, (\gamma_{p'})_{l})
\le C_0^3D_{max}
\end{equation}

We will use the notation $\vec{\xi}^{(i)}:=(\xi_1^{(i)},\ldots,\xi_k^{(i)})\in\R^k$. Recall the functions $G_i$ from the introduction. Consider now the following system of ($k$ vector, or equivalently $2k$ scalar) linear equations in $2k$ variables $\xi_1^{(1)},\ldots,\xi_k^{(2)}\in\R$, and coefficients $v_j\in\R^2$.

\begin{equation}
\label{solvablesystem}
\begin{cases} G_1(\vec{\xi}^{(1)})-G_1(\vec{\xi}^{(2)})=v_1
\\ G_2(\vec{\xi}^{(1)})=v_2\\ \ldots \\ G_{m}(\vec{\xi}^{(1)})=v_{m}\\ G_{m+1}(\vec{\xi}^{(2)})=v_{m+1}\\ \ldots \\ G_{k}(\vec{\xi}^{(2)})=v_{k}
\end{cases}
\end{equation}
By hypothesis, the system has a unique solution for each $v_i$.  Let now $\vec{\xi}^{(1)},\vec{\xi}^{(2)}\in \R^k$  consist of the first $k$ entries of $\gamma_{s}-\gamma_{p}$ and $\gamma_{s'}-\gamma_{p'},$ respectively,
and let $v_j$ be the values corresponding to this choice in the system above. Note  that for each $j\in\{1,\ldots,k\}$
\begin{equation}
\label{ast89r34riokfgvknjcguyt}
|v_j|\le 2C_0^3D_{max},
\end{equation}
by \eqref{solvablesystemsahhdty823r3409rftk obrjhituhi}. It will follow that the solution $(\vec{\xi}^{(1)},\vec{\xi}^{(2)})$ will be bounded in norm by the norm of $(v_1,\ldots,v_k)$ times a constant that only depends on the coefficients of the linear forms $G_1,\ldots,G_k$. Thus, if $C_0$ is chosen large enough compared to these coefficients, we get that $\dist(\gamma_{s},\gamma_{p}),\dist(\gamma_{s'},\gamma_{p'})\le C_0^4D_{max}$. Now (r7) is immediate.

There will be two distinct stages, each of which will generate some vector trees. In each stage, before we construct the vector trees we will have to carefully reshuffle the collections $\P_i$.

Let us describe the first stage of the construction. We will first aim at separating the trees in each family $\F_i$, and to achieve this we will employ a trick first used by Fefferman in \cite{Fe}. Fix $1\le i\le n$. For each $l\ge 0$, let
$$\P_i(l):=\{s_i\in\P_i:2^l\le |\{\T\in\F_i:\xi_{\T}\in \tilde{\omega}_{s_i},\;R_{s_i}\subseteq R_{\T}\}|<2^{l+1}\}.$$
Note that $(\P_i(l))_{l\ge 0}$ forms a partition of $\P_i$. Next, we organize each $\P_i(l)$ into  $i-$ trees with top tiles. More precisely, consider the collection of all tiles  $\P_i^*(l)\subseteq \P_i(l)$ which are maximal with respect to the order relation $\lesssim$.

We observe a few things. First, the tiles in $\P_i^{*}(l)$ are pairwise not comparable under $\lesssim$. Second, for each tile  $s_i\in\P_i(l)$ there is a {\em unique} tile $s_i^{*}\in\P_i^{*}(l)$ such that $s_i\lesssim s_i^{*}$. To see the uniqueness part, assume by contradiction
that $s_i\lesssim s_i^{*}$ and $s_i\lesssim s_i^{**}$ for some $s_i^{*}\not=s_i^{**}\in\P_i^{*}(l)$. Then $R_{s_i^{*}}\cap R_{s_i^{**}}\not=\emptyset$, which forces $\tilde{\omega}_{s_i^{*}}\cap \tilde{\omega}_{s_i^{**}}=\emptyset$. This together with the fact that  $s_i^*,s_i^{**}\in \P_i(l)$ will imply that
$$|\{\T\in\F_i:\xi_{\T}\in \tilde{\omega}_{s_i},\;R_{s_i}\subseteq R_{\T}\}|\ge 2\times2^{l},$$
which forces the contradiction $s_i\in \bigcup_{l'\ge l+1}\P_i(l')$.

Now, for each $s_i^{*}\in \P_i^{*}(l)$ we form the $i-$tree with top $(R_{s_i^{*}},c(\omega_{s_i^{*}}))$ consisting of all tiles in $\P_i(l)$ which are $\lesssim s_i^{*}$. We have just seen that these trees partition $\P_i(l)$ and that tiles in distinct trees are not comparable under $\lesssim$. Call the collection of these trees $\F_i(l)$. It is easy to see that
\begin{equation}
\label{countingfstaysthesame}
N_{\F_i(l)}(\vec{x})\le N_{\F_i}(\vec{x}),\;\;\vec{x}\in\R^2.
\end{equation}

We now use these trees to build our first generation of vector trees. For a moment we will abuse notation and for two $s,s'\in\P$ we will write $s\lesssim s'$ if $s_i\lesssim s_i'$ for each $1\le i\le n$. Note that the extension of $\lesssim$ from tiles to the multi-tiles in  $\P$ remains a relation of order.

For each $\vec{l}:=(l_1,\ldots,l_n)$ with $l_1,\ldots,l_n\ge 0$ denote
$$\P(\vec{l}):=\{s\in\P:s_i\in \P_i(l_i),\;\hbox{for each}\;1\le i\le n\}.$$
The selection process goes as follows. Fix $\vec{l}$. Find a maximal (with respect to $\lesssim$) $s\in \P(\vec{l})$, then construct the vector tree $\vec{\T}(s)$ with top $(R_{s},c(\omega_s))$ consisting of all $s'\in \P(\vec{l})$ such that $s'\lesssim s$. Then eliminate all multi-tiles in $\vec{\T}(s)$ from $\P(\vec{l})$, and restart the selection (with this new value for $\P(\vec{l})$). When no such vector tree remains to be selected,  the value of $\P(\vec{l})$ will be $\emptyset$. Denote by $\F(\vec{l})$ the family of the vector trees selected at this stage, and by $\F^{**}(\vec{l})\subset \F(\vec{l})$ those vector trees which consist of at least two multi-tiles (i.e., in addition to their top, they also contain a multi-tile with a scale distinct from the scale of the top). The vector trees in $\F(\vec{l})\setminus\F^{**}(\vec{l})$ will  be reshuffled later, so we will  ignore them for the moment.

We will first show how to control the counting function $N_{\F^{**}(\vec{l})}$ in terms of each $N_{\F_i(l_i)}$. For each $\vec{\T}\in \F(\vec{l})$ we denote with $\T_i$ the projection of $\vec{\T}$ onto $\S_i$ (and this is an $i-$tree). For each $\vec{x}\in\R^d$, denote by $\F^{**}(\vec{l},\vec{x})$ the collection of all vector trees  $\vec{\T}\in\F^{**}(\vec{l})$ such that $\vec{x}\in R_{\vec{\T}}$. A similar definition holds for $\F_i(l_i,\vec{x})$.

Let us first make two easy observations. On the one hand, note that for each $\vec{\T}\in \F(\vec{l})$, all the tiles of $\T_i$ are contained in a unique tree from the family $\F_i(l_i)$ (this follows from an earlier observation, and from the fact that the tiles in $\T_i$ are pairwise comparable under $\lesssim$). We will refer to this tree as the $i^{th}$ {\em standard projection} of $\vec{\T}$. On the other hand, for a fixed $\vec{x}$ and some $\vec{\T}\in \F^{**}(\vec{l},\vec{x})$, if for each $1\le i\le n$ we know the  $i^{th}$ standard projection of $\vec{\T}$, we will also automatically know $\vec{\T}$ (this follows from the maximality involved in the selection of $\vec{\T}$). We will see in Proposition \ref{knowmknowall} that more is true, namely that knowledge of just $m$ of the standard projections suffices to determine $\vec{\T}$.

At this point we recall the following lemma from \cite{KT}.

\begin{lemma}
\label{lemakatztao}
Let $X$ and $A$ be finite sets and let $g:X\to A$ be a function. Then
$$|\{(x_1,x_2)\in X^{2}:g(x_1)=g(x_2)\}|\ge \frac{|X|^{2}}{|A|}.$$
\end{lemma}

We will apply this lemma with $X=\F^{**}(\vec{l},\vec{x})$, $A=\F_1(l_1,\vec{x})$ while $g(\vec{\T}):=f_1(\vec{\T})$ is the first standard projection of $\vec{\T}$ (for later use,  we extend this definition to all $f_i$, $1\le i\le n$). We get that
\begin{equation}
\label{sbaweud0ofrg;,gbljfduicyewi4u}
|\{(\vec{\T_1},\vec{\T_2})\in X^2:f_1(\vec{\T}_{1})=f_1(\vec{\T}_2)\}|\ge \frac{[N_{\F^{**}(\vec{l})}(\vec{x})]^2}{N_{\F_1(l_1)}(\vec{x})}.
\end{equation}

Next, we will estimate from above the size of the set
$$\H:=\{(\vec{\T_1},\vec{\T_2})\in X^2:f_1(\vec{\T}_{1})=f_1(\vec{\T}_2)\}.$$
In particular, we will show that the function
$$H:\H\to \prod_{j=2}^{k}\F_j(l_j,\vec{x})$$
defined by
\begin{equation}
\label{defH}
H(\vec{\T_1},\vec{\T_2})=(f_2(\vec{\T}_1),\ldots,f_{m}(\vec{\T}_1), f_{m+1}(\vec{\T}_2),\ldots,f_{k}(\vec{\T}_2))
\end{equation}
is injective (recall that $k/d<n/2$, thus $k<n$, and the functions $f_i$ will make sense for each $1\le i\le k$). This fact combined with \eqref{sbaweud0ofrg;,gbljfduicyewi4u} will lead to the desired estimate
\begin{equation}
\label{estforcountfunction1}
N_{\F^{**}(\vec{l})}(\vec{x})\le \left(\prod_{j=1}^{k}N_{\F_j(l_j)}(\vec{x})\right)^{1/2}.
\end{equation}
Note that the sum of the exponents on the right hand side equals the rank $k/2$ (and this is the best one can do). It will be crucial that this number is $<n/2$.  Similarly, in the case of general $d$, one can arrange things such that the sum of the exponents will be\footnote{It seems likely that one can achieve an inequality where the sum of the exponents is $k/d$. However, we will content ourselves with a sum barely smaller than $n/2$.} $<n/2$. This is explained in the next section.

By using this, \eqref{countingfstaysthesame} and the fact that for each $\vec{x}\in\R^2$,
\begin{equation}
\label{wdwer4t6hestforcountfunction1}
N_{\F_i(l_i)}(\vec{x})=0\hbox{ if }2^{l_i}>N_{\F_i}(\vec{x}),
\end{equation}
 we find that\footnote{The extra $\epsilon$ exponent  hides a logarithmic gain} for each $\epsilon>0$
\begin{equation*}
N_{\F^{**}}(\vec{x})\lesssim_{\epsilon} \left(\prod_{j=1}^{k}N_{\F_j}(\vec{x})\right)^{1/2+\epsilon}.
\end{equation*}
Here and in the following
$$\F^{**}:=\bigcup_{\vec{l}}\F^{**}(\vec{l}).$$
Also, since there was nothing special about working  with indices $1,\ldots,k$, we can permute them and get similar inequalities. Combining this with the fact that $k<n$, we get the final inequality
\begin{equation}
\label{estforcountfunction1fgfdgfh}
N_{\F^{**}}(\vec{x})\lesssim \left(\prod_{j=1}^{n}N_{\F_j}(\vec{x})\right)^{\delta},
\end{equation}
for some $\delta<1/2$, depending only on $k$ and $n$. The precise value of $\delta$ will not be important, what matters for later applications is the fact that it is strictly smaller than $1/2$.

Let us now prove the following:

\begin{proposition}
\label{p:skhxghduyetd78436r2843908}
The function $H$ is injective.
\end{proposition}
\begin{proof}
Assume that
$$(\vec{\T}_1,\vec{\T}_2),(\vec{\V}_1,\vec{\V}_2)\in \H$$
have the same value under $H$. We will prove that $(\vec{\T}_1,\vec{\T}_2)=(\vec{\V}_1,\vec{\V}_2)$.
Let $s,s',p.p'$ the top multi-tiles of the vector trees $\vec{\T}_1,\vec{\T}_2,\vec{\V}_1,\vec{\V}_2$ and let $D_{max}$ be as in \eqref{t46e7wncgd}. Our hypothesis will easily imply that
\begin{equation}
\label{194574857889fdhjhgjfhjrtr7863436r3hgf}
D_{max}\le C_0\max\{\diam(\tilde{\omega}_s),\diam(\tilde{\omega}_{s'}),\diam(\tilde{\omega}_p),\diam(\tilde{\omega}_{p'})\},
\end{equation}
and by (r7) we get
\begin{equation}
\label{fhjdhferuifyrfy748685656097}
\dist(\omega_s,\omega_p), \dist(\omega_{s'},\omega_{p'})\le C_0^6\max\{\diam(\tilde{\omega}_s),\diam(\tilde{\omega}_{s'}),\diam(\tilde{\omega}_p),\diam(\tilde{\omega}_{p'})\}.
\end{equation}

We choose one of the four vector trees for which its top multi-tile has the largest scale of the frequency component. To fix notation, we can assume without any loss of generality that this vector tree is $\vec{\T}_1$. Since $\vec{\T}_1$ has at least two multi-tiles with distinct scales, we can find $t\in \vec{\T}_1$ such that $\diam(\tilde{\omega}_t)>\diam(\tilde{\omega}_s)$. This observation combined with \eqref{fhjdhferuifyrfy748685656097}, (r5) and the fact that $C_0,C_1<<C_2$ implies that $\tilde{\omega}_p\subset \tilde{\omega}_t$. Since $R_t\subset R_{s}$ and since $R_s\subset R_p$ (this being a consequence of the fact that $\vec{x}\in R_s\cap R_p$ and the fact that the scale of $R_p$ is larger than the scale of $R_s$), it follows that $t\lesssim p$. But we also know that $t\lesssim s$, and thus $\vec{\T}_1$ and $\vec{\V}_1$ will share all standard projections. This forces $\vec{\T}_1=\vec{\V}_1$. This will in turn imply that $\vec{\T}_2$ and $\vec{\V}_2$ share  at least $m$ standard projections (corresponding to the indices $j\in\{1,m+1,m+2,\ldots,k\}$; recall that $2m-1\le k$). The fact that $\vec{\T}_2=\vec{\V}_2$ will follow from the following proposition, which is somewhat reminiscent of the rank property (iv).
\end{proof}

\begin{proposition}
\label{knowmknowall}
Let $\vec{\T}\in\F^{**}(\vec{l},\vec{x})$. If we know the value of $f_i(\vec{\T})$ for $m$ of the $n$ values of $i$, then we know $\vec{\T}$.
\end{proposition}
\begin{proof}
Let $\vec{\T},\vec{\V}\in \F^{**}(\vec{l},\vec{x})$ such that $f_i(\vec{T})=f_i(\vec{\V})$ for each $i\in\{1,\ldots,m\}$. Follow exactly the same approach as in the proof of Proposition \ref{p:skhxghduyetd78436r2843908}, this time applied to the pairs $(\vec{\T},\vec{\T})$ and $(\vec{\T},\vec{\V})$.
\end{proof}

At this point the proof of \eqref{estforcountfunction1} is complete, and the first stage of our construction ends.

In the second stage, we will have to deal with the collections of vector trees $\F^{*}(\vec{l}):=\F(\vec{l})\setminus\F^{**}(\vec{l})$ each of which consists of just one multi-tile. We will denote by $\bar{\P}(\vec{l})$ the collection of all these multi-tiles. The additional key property that these multi-tiles will obey is
\begin{equation}
\label{keyproprthy}
s\not=s'\in \bar{\P}(\vec{l})\implies s,s'\hbox{ are not comparable under }\lesssim.
\end{equation}
This is a consequence of the maximality involved in the selection of the vector trees from the previous stage. \eqref{keyproprthy} will turn out to be crucial in proving the analog of Proposition \ref{knowmknowall},  see Proposition \ref{p:skhxghduyetd78436r2843908secondversion} below.

 The plan is to reshuffle $\bar{\P}(\vec{l})$ into convenient vector trees. To achieve this, we will first split each $\bar{\P}_i(\vec{l})$ into overlapping trees (interestingly, from now on, the lacunary trees will not play any role). For each $r\in\Z$ with $2^{-r}\le \size_{i}(\bar{\P}_i(\vec{l}))$, let
\begin{equation}
\label{defertfedderdedrdsllkjj}
\bar{\P}_i^{(r)}(\vec{l}):=\{s_i\in \bar{\P}_i(\vec{l}):2^{-r}\le|R_{s_i}|^{-1/2}|\langle F_i,\phi_{s_i}\rangle|<2^{-r+1}\}.
\end{equation}
Note that
$$\bar{\P}_i(\vec{l})=\bigcup_{2^{-r}\le \size_{i}(\bar{\P}_i(\vec{l}))}\bar{\P}_i^{(r)}(\vec{l}).$$

We next use a greedy selection algorithm as before to split each $\bar{\P}_i^{(r)}(\vec{l})$ into a collection $\F_i(r,\vec{l})$ of disjoint $i-$overlapping trees with top tiles pairwise not comparable with respect to $\le$. This implies via a classical $TT^*$ argument (see for example Corollary 7.6. in \cite{MTT1}) that
\begin{equation}
\label{anotherBesselinBMO}
\|N_{\F_i(r,\vec{l})}(x)\|_{\BMO}\lesssim 2^{2r}[\sup_{s_i\in\bar{\P}_i^{(r)}(\vec{l})}\inf_{\vec{x}\in R_{s_i}}M_2F_i(\vec{x})]^2.
\end{equation}
Also, \eqref{defertfedderdedrdsllkjj} will imply that $N_{\F_i(r,\vec{l})}$ is supported in the set $\{\vec{x}:M_2F_i(\vec{x})\gtrsim 2^{-r}\}$. Combining this with \eqref{anotherBesselinBMO} and John-Nirenberg's inequality we get
\begin{equation}
\label{anotherBesselinLq}
\|N_{\F_i(r,\vec{l})}(x)\|_{t}\lesssim 2^{2r}[\sup_{s_i\in\bar{\P}_i^{(r)}(\vec{l})}\inf_{\vec{x}\in R_{s_i}}M_2F_i(\vec{x})]^2[2^{p_ir}\|F_i\|_{p_i}^{p_i}]^{1/t},\;1\le t<\infty.
\end{equation}

We then apply Fefferman's trick again, as in the first stage of our construction, this time however with respect to $\F_i(r,\vec{l})$ and $\le$ (rather than $\lesssim$).

For each $q\ge 0$, let
$$\bar{\P}_i^{(r)}(\vec{l},q):=\{s_i\in\bar{\P}_i^{(r)}(\vec{l}):2^q\le |\{\T\in\F_i(r,\vec{l}):\xi_{\T}\in \omega_{s_i},\;R_{s_i}\subseteq R_{\T}\}|<2^{q+1}\}.$$
Note that $(\bar{\P}_i^{(r)}(\vec{l},q))_{q\ge 0}$ forms a partition of $\bar{\P}_i^{(r)}(\vec{l})$. Next, we organize each $\bar{\P}_i^{(r)}(\vec{l},q)$ into $i-$overlapping trees with top tiles. More precisely, consider the collection $\bar{\P}_i^{(r,*)}(\vec{l},q)\subset \bar{\P}_i^{(r)}(\vec{l},q)$ of all tiles which are maximal with respect to the order relation $\le$.

It follows as before that the tiles in $\bar{\P}_i^{(r,*)}(\vec{l},q)$ are pairwise  not comparable under $\le$ and that for each tile  $s_i\in\bar{\P}_i^{(r)}(\vec{l},q)$ there is a unique tile $s_i^{*}\in\bar{\P}_i^{(r,*)}(\vec{l},q)$ such that $s_i\le s_i^{*}$.

Now, for each $s_i^{*}\in \bar{\P}_i^{(r,*)}(\vec{l},q)$ we form the $i-$overlapping tree with top $(R_{s_i^{*}},c(\omega_{s_i^{*}}))$ consisting of all tiles in $\bar{\P}_i^{(r)}(\vec{l},q)$ which are $\le s_i^{*}$. As before,  these trees partition $\bar{\P}_i^{(r)}(\vec{l},q)$ and  tiles in distinct trees are not comparable under $\le$. Call the collection of these trees $\F_i(r,\vec{l},q)$, and note that
\begin{equation}
\label{chvhcggyeuiwqywqet36e23dejkjrgutuio}
N_{\F_i(r,\vec{l},q)}(\vec{x})\le N_{\F_i(r,\vec{l})}(\vec{x}),\;\;\vec{x}\in\R^2.
\end{equation}

For each vectors $\vec{q}$ and $\vec{r}$ let $\F(\vec{r},\vec{l},\vec{q})$ consist of all the multi-tiles $s$ with $s_i\in \bar{\P}_i^{(r_i)}(\vec{l},q_i)$ for each $i$. This consideration is motivated by \eqref{keyproprthy}, which implies that each vector tree in $\bar{\P}(\vec{l})$ can have only one multi-tile. In the following, we will prove that
\begin{equation}
\label{chvhcggyeuiwqyduy3e8394945056=-067}
N_{\F(\vec{r},\vec{l},\vec{q})}(\vec{x})\le \left(\prod_{i=1}^{k}N_{\F_i(r_i,\vec{l},q_i)}(\vec{x})\right)^{1/2},\;\;\vec{x}\in\R^2.
\end{equation}
Using this, \eqref{chvhcggyeuiwqywqet36e23dejkjrgutuio} and the fact that for each $\vec{x}\in\R^2$, $N_{\F_i(r_i,\vec{l},q_i)}(\vec{x})=0$ if $2^{q_i}>N_{\F_i(r_i,\vec{l})}(\vec{x})$, we find that for each $\epsilon>0$
\begin{equation*}
N_{\F(\vec{r},\vec{l})}(\vec{x})\lesssim_{\epsilon}\left(\prod_{j=1}^{k}N_{\F_i(r_i,\vec{l})}(\vec{x})\right)^{1/2+\epsilon},\;\;\vec{x}\in\R^2.
\end{equation*}
Here and in the following,
$$\F(\vec{r},\vec{l}):=\bigcup_{\vec{q}}\F(\vec{r},\vec{l},\vec{q}).$$
Again, by permuting indices and since $k<n$, we get
\begin{equation}
\label{chvhcggyeuiwqyduy3e8394945056=-067asxwqr3466uuk}
N_{\F(\vec{r},\vec{l})}(\vec{x})\lesssim\left(\prod_{i=1}^{n}N_{\F_i(r_i,\vec{l})}(\vec{x})\right)^{\delta},\;\;\vec{x}\in\R^2,
\end{equation}
for some $\delta<1/2$.
We mention  that \eqref{chvhcggyeuiwqyduy3e8394945056=-067asxwqr3466uuk} will later be used in conjunction with \eqref{anotherBesselinLq} and with the upper bound \eqref{wdwer4t6hestforcountfunction1} on the $l_i$.

We next prove \eqref{chvhcggyeuiwqyduy3e8394945056=-067}. The argument follows exactly the same scheme as in the previous stage of the decomposition, and we sketch it briefly. We denote by
$X:=\F(\vec{r},\vec{l},\vec{q},\vec{x})$ the collection of all vector trees  $\vec{\T}\in\F(\vec{r},\vec{l},\vec{q})$ with $\vec{x}\in R_{\vec{\T}}$. For a vector tree $\vec{\T}$ we will now denote by $f_i(\vec{\T})$ the tree from the collection $\F({r_i},\vec{l},{q_i},\vec{x})$ which contains the tiles $s_i$, for $s\in\vec{\T}$. It will follow as in the previous stage of the construction that these projections are well defined. Define $\H$ and $H$ as before, more precisely
$$H:\H\to \prod_{j=2}^{k}\F_{j}(r_{j},\vec{l},q_{j},\vec{x}).$$
\begin{proposition}
\label{p:skhxghduyetd78436r2843908secondversion}
The function $H$ is injective.
\end{proposition}
\begin{proof}
The proof of the injectivity of $H$  follows the same lines as the one of Proposition \ref{p:skhxghduyetd78436r2843908}, with only one key difference.
Assume that
$$(\vec{\T}_1:=\{s\},\vec{\T}_2:=\{s'\}),(\vec{\V}_1:=\{p\},\vec{\V}_2:=\{p'\})\in \H$$
have the same value under $H$. We will prove that $(s,s')=(p,p')$.
Let $D_{max}$ be as in \eqref{t46e7wncgd}. Our hypothesis will easily imply that
$$D_{max}\le C_0\max\{\diam(\omega_s),\diam(\omega_{s'}),\diam(\omega_p),\diam(\omega_{p'})\}.$$
This estimate is stronger than the one in \eqref{194574857889fdhjhgjfhjrtr7863436r3hgf}, in that on the right hand side here we have the diameters of the cubes $\omega$, rather than those of the cubes $\tilde{\omega}$. This is due to the fact that the standard projections now reflect positioning with respect to $\le$, rather than $\lesssim$. By (r7) we get
\begin{equation}
\label{fhjdhferuifyrfy748685656097secondv}
\dist(\omega_s,\omega_p), \dist(\omega_{s'},\omega_{p'})\le C_0^6\max\{\diam(\omega_s),\diam(\omega_{s'}),\diam(\omega_p),\diam(\omega_{p'})\}.
\end{equation}

We choose one of the four vector trees whose frequency component has the largest scale. To fix notation, we can assume without any loss of generality that this vector tree is $s$.  Observe that \eqref{fhjdhferuifyrfy748685656097secondv} shows that $\tilde{\omega}_s$ and $\tilde{\omega}_p$ must intersect, if $C_0<<C_1$. But then (r5) implies that $\tilde{\omega}_p\subseteq \tilde{\omega}_s$. This combined with the fact that $\vec{x}\in R_s\cap R_p\not=\emptyset$ implies that $s\lesssim p$. From \eqref{keyproprthy} we immediately get that $s=p$. This in turn implies that $s'$ and $p'$ share at least $m$ standard projections. This is equivalent with saying that $s'_i\le p'_i$ (or vice versa), for at least $m$ values of $i$. From (r2) we get that $s'\lesssim p'$ (or vice versa). A final invocation of \eqref{keyproprthy} concludes that $s'=p'$.
\end{proof}

\subsection{The case $d>2$}
\label{secd>2}

The case of arbitrary $d$ follows by considerations very similar to the ones involved in the case $d=2$. The rank properties (very much in spirit like (r7)), that will be needed throughout the proof will not be stated explicitly this time, but will rather become clear from the non-degeneracy assumptions on $\Gamma'$ that will be stated in each case.
We briefly sketch the details.

Recall we are under the assumption $m\ge 2$. We can in addition assume that $m\ge n/2$ (and thus $k/d<m$). This is because if $m<n/2$, then, as explained earlier, one can crudely treat the operator as having integral rank $m$, and apply the methods from \cite{MTT1} (or alternatively, the approach described in Section \ref{$d=1$}).

We will need the general case of the combinatorial lemma from \cite{KT}.

\begin{lemma}
\label{./x,.ewuf490389}
Let $X$ and $A_1,\ldots,A_{d-1}$ be finite sets and for each $1\le i\le d-1$ let $g_i:X\to A_i$ be a function. Then
$$|\{(x_1,\ldots,x_{d})\in X^{d}:g_i(x_{i})=g_i(x_{i+1})\hbox{ for all }1\le i\le d-1\}|\ge \frac{|X|^{d}}{\prod_{i=1}^{d-1}|A_i|}.$$
\end{lemma}

The two stages of the construction are the same as in the case $d=2$, but the choice of the function $H$ involves some modifications. As in the case $d=2$, $H$ will have the same formula in both stages of the reshuffling process, and proving its injectivity will involve very similar ideas. Thus, to fix notation, we only sketch the argument corresponding to the first stage. Take $X=\F^{**}(\vec{l},\vec{x})$. Note that $m\le n+1$ (since $m\le \frac{n+1}{2}$ and $n\ge 3$), so the sets $\F_j(l_j,\vec{x})$ are defined for each $1\le j\le m+1$.
We distinguish three separate cases, and will address each of them below.

\textbf{Case 1: $n$ is even}. The first case we describe is when $n$ is even. It follows that $m\le \frac{n}{2}$.  Without loss of generality we can assume that
\begin{equation}
\label{dfweytyrfw478rt7e47fjkdhjvjhy}
\max\{|\F_j(l_j,\vec{x})|:1\le j\le m\}=|\F_1(l_1,\vec{x})|\le\min\{|\F_j(l_j,\vec{x})|:m+1\le j\le n\}.
\end{equation}
Apply Lemma \ref{./x,.ewuf490389} with $A_j=\F_1(l_1,\vec{x})$, $g_j=f_1$. Also, define
$$\H:=\{(\vec{\T_1},\ldots,\vec{\T_d})\in X^d:g_j(\vec{\T}_{j})=g_j(\vec{\T}_{j+1}),1\le j\le d-1\},$$
$$H:\H\to \prod_{j}B_j,$$
where each $B_j$ equals one of the sets $\F_i(l_i,\vec{x})$, and
$$H(\vec{\T}_1,\ldots,\vec{\T}_d):=$$ $$(f_2(\vec{\T}_1),f_3(\vec{\T}_1),\ldots,f_{m}(\vec{\T}_1),f_{2}(\vec{\T}_2),f_{3}(\vec{\T}_2),\ldots,f_{m}(\vec{\T}_2),\ldots,
f_2(\vec{\T}_d),f_3(\vec{\T}_d),\ldots,f_{m}(\vec{\T}_d))$$

We briefly comment on this construction. It is one of many one can do, and while the non-degeneracy requirements to make a particular choice of $H$ injective will depend on $H$, they are achieved for generic $\Gamma'$. For example, we chose to assign entries of the form $f_2,\ldots,f_m$ to each of the trees $\vec{\T}_j$, but we could also have chosen instead to alternate between these entries and the entries  $f_{m+1},\ldots,f_{2m-1}$, as we did in the case $d=2$. The only restrictions are that the entries corresponding to each tree $\vec{\T}_j$ are  pairwise distinct (otherwise redundancy occurs), and that the entries for $\vec{\T}_j$ are also distinct from  $f_1(\vec{\T}_j)$.

Let us assume for the moment that $H$ is injective. Combining this with Lemma \ref{./x,.ewuf490389} we get that
$$|\F^{**}(\vec{l},\vec{x})|\le (|\F_1(l_1,\vec{x})|^{d-1}\prod_{j=2}^{m}|\F_j(l_j,\vec{x})|)^{1/d}.$$
Due to our assumption \eqref{dfweytyrfw478rt7e47fjkdhjvjhy}, one can easily check that this implies
\begin{equation}
\label{dfbrhgf3476r281390434dl;e,cerl;gio5=-y0456y7hlty,nl}
|\F^{**}(\vec{l},\vec{x})|\le (\prod_{j=1}^{n}|\F_j(l_j,\vec{x})|)^{\delta},
\end{equation}
for some $\delta<1/2$, which as explained earlier, is the desired inequality.

Let us now see why $H$ is injective. Assume that $(\vec{\T}_1,\ldots,\vec{\T}_d),(\vec{\V}_1,\ldots,\vec{\V}_d)\in\H$ have the same value under $H$. The non-degeneracy condition that we need is that the following system
\begin{equation}
\label{solvablesystemsecondversionttt}
\begin{cases} G_1(\vec{\xi}^{(1)})-G_1(\vec{\xi}^{(2)})=v_1
\\ G_1(\vec{\xi}^{(2)})-G_1(\vec{\xi}^{(3)})=v_2
\\ \dots\ldots\ldots
\\ G_{1}(\vec{\xi}^{(d-1)})-G_1(\vec{\xi}^{(d)})=v_{d-1}
\\ G_2(\vec{\xi}^{(1)})=v_{1,2}\\ \ldots\ldots\ldots \\ G_{m}(\vec{\xi}^{(d)})=v_{d,m}
\end{cases}
\end{equation}
has {\em at most one solution}, for each choice of targets $v_1,\ldots,v_{d-1}$ and $v_{i,j}$, $1\le i\le d$, $2\le j\le m$. Note that there are $dm-1$ $\R^d$ valued equations in $d$ unknowns $\vec{\xi}^{(j)}\in\R^k$, and that we have $(dm-1)d\ge dk$, since $m>\frac{k}{d}$. Thus the system above is always overdetermined, and we require that the matrix associated with it has maximum rank $dk$. It is not too hard to check that this is achieved for a generic choice of $\Gamma'$. Actually, our choice of $H$ is in such a way, that the system consisting of the first $k$ vector valued equations above will generically give rise to a compatible system.

The injectivity of $H$ now follows as in the case $d=2$. First, there must be some $i_0$ such that either $\vec{\T}_{i_0}$ or $\vec{\V}_{i_0}$ contains a multi-tile with the frequency  scale larger than or equal to the scales of all the multi-tiles from the trees $\vec{\T}_{i}, \vec{\V}_{i}$. As before, we get that $\vec{\T}_{i_0}=\vec{\V}_{i_0}$. The equality $\vec{\T}_{i}=\vec{\V}_{i}$ for the remaining indices $i$ will follow from a "domino effect". It first follows for the neighboring indices $i=i_0-1$ and/or $i=i_0+1$, using the fact that $f_1(\vec{\T}_{i})=f_1(\vec{\T}_{i_0})=f_1(\vec{\V}_{i_0})=f_1(\vec{\V}_{i})$, $f_j(\vec{\T}_{i})=f_j(\vec{\V}_{i})$ (for $j\in\{2,3,\ldots,m\}$), and using Proposition  \ref{knowmknowall}. The domino effect continues until all indices are covered.

\textbf{Case 2: $n$ is odd}. Since we have assumed that $m\ge n/2$,  it follows that $m=\frac{n+1}{2}$. In particular, we observe that whenever $i_1\notin\{i_2,\ldots,i_{2m-1}\}$ with $\{i_2,\ldots,i_m\}$ pairwise distinct and $\{i_{m+1},\ldots,i_{2m-1}\}$ pairwise distinct, the following system with $n=2m-1$ $\R^d$ valued equations in 2 unknowns $\vec{\xi}^{(1)}, \vec{\xi}^{(2)}\in \R^k$
\begin{equation}
\label{solvablesystemsecondversionhjuy}
\begin{cases} G_{i_1}(\vec{\xi}^{(1)})-G_{i_1}(\vec{\xi}^{(2)})=v_1
\\ G_{i_2}(\vec{\xi}^{(1)})=v_{1,2}\\
\ldots\ldots\ldots
\\G_{i_{m}}(\vec{\xi}^{(1)})=v_{1,m}
\\ G_{i_{m+1}}(\vec{\xi}^{(2)})=v_{2,m+1}\\
\ldots\ldots\ldots
\\ G_{i_{2m-1}}(\vec{\xi}^{(2)})=v_{2,2m-1}
\end{cases}
\end{equation}
will have at most one solution for a generic $\Gamma'$.
This is one of the two non-degeneracy conditions that will be needed in both of the following two subcases. We will refer to the above system as a {\em two-scheme}.

\textbf{Subcase 2a: $n$ is odd and $d$ is odd}.

Without loss of generality we can assume that
\begin{equation}
\label{dfweytyrfw478rt7e47fjkdhjvjhyfhjvghfkjghjkjgkjgkgjklgjkljthioihohirooptogp[ogp[to}
\max\{|\F_j(l_j,\vec{x})|:1\le j\le m+1\}=|\F_2(l_2,\vec{x})|\le\min\{|\F_j(l_j,\vec{x})|:m+2\le j\le n\}.
\end{equation}
Apply Lemma \ref{./x,.ewuf490389} with $A_j=\F_1(l_1,\vec{x})$ and $g_j=f_1$ if $j$ is odd, $A_j=\F_2(l_2,\vec{x})$ and $g_j=f_2$ if $j$ is even. Also, define
$$\H:=\{(\vec{\T_1},\ldots,\vec{\T_d})\in X^d:g_j(\vec{\T}_{j})=g_j(\vec{\T}_{j+1}),1\le j\le d-1\},$$
$$H:\H\to \prod_{j}B_j,$$
where each $B_j$ equals one of the sets $\F_i(l_i,\vec{x})$, and
$$H(\vec{\T}_1,\ldots,\vec{\T}_d):=$$
$$(f_3(\vec{\T}_1),f_4(\vec{\T}_1),\ldots,f_{m}(\vec{\T}_1),f_{3}(\vec{\T}_2),f_{4}(\vec{\T}_2),\ldots,f_{m}(\vec{\T}_2),\ldots,
f_3(\vec{\T}_d),f_4(\vec{\T}_d),\ldots,f_{m}(\vec{\T}_d),$$
$$f_{m+1}(\vec{\T}_1),f_{m+1}(\vec{\T}_3),\ldots,f_{m+1}(\vec{\T}_d)).$$
Note that $H$ has $(m-2)d+\frac{d+1}{2}$ entries. If $m=2$, then the first $(m-2)d$ entries are not present. Note also that the last line above contains $\frac{d+1}{2}$ entries of the form $f_{m+1}(\vec{\T}_i)$, for all possible odd indices $i$. Here is why we can not use more than $\frac{d+1}{2}$ such entries. Why more entries would certainly reinforce the injectivity of $H$, the application of the injectivity (combined with Lemma \ref{./x,.ewuf490389}) would be inefficient, in that it would not lead to \eqref{dfbrhgf3476r281390434dl;e,cerl;gio5=-y0456y7hlty,nl}. On the other hand, our choice for $H$ combined with the assumption \eqref{dfweytyrfw478rt7e47fjkdhjvjhyfhjvghfkjghjkjgkjgkgjklgjkljthioihohirooptogp[ogp[to} is easily seen to guarantee \eqref{dfbrhgf3476r281390434dl;e,cerl;gio5=-y0456y7hlty,nl}.

The reason we chose to assign entries $f_{m+1}(\vec{\T}_i)$ to the odd indices $i$ (as opposed to -say- the first $\frac{d+1}{2}$ indices) is to allow for the domino effect, as explained below.

Note also that in the definition of $g_j$ we chose to alternate between $f_1$ and $f_2$, to prevent certain redundancies from occurring. This will become clear in a moment.

In addition to the requirement that two-scheme \eqref{solvablesystemsecondversionhjuy} has at most one solution, we will also require that the system
\begin{equation}
\label{solvablesystemsecondversiontttsdute76rt346r43fjhgrgrg}
\begin{cases} G_1(\vec{\xi}^{(1)})-G_1(\vec{\xi}^{(2)})=v_1
\\ G_2(\vec{\xi}^{(2)})-G_2(\vec{\xi}^{(3)})=v_2
\\ \dots\ldots\ldots
\\ G_{2}(\vec{\xi}^{(d-1)})-G_{2}(\vec{\xi}^{(d)})=v_{d-1}
\\ G_{l}(\vec{\xi}^{(j)})=v_{j,l},\;\;3\le l\le m,\;1\le j\le d
\\ G_{m+1}(\vec{\xi}^{(1)})=v_{1,m+1}\\
\\ G_{m+1}(\vec{\xi}^{(3)})=v_{3,m+1}\\
\ldots\ldots\ldots \\
\\ G_{m+1}(\vec{\xi}^{(d)})=v_{d,m+1}
\end{cases}
\end{equation}
has at most one solution. Note that there are $d(m-\frac{1}{2})-\frac12$ $\R^d$ valued equations and $d$ unknowns in $\R^k$. Our assumption that $k/d<m$ will imply that the system is overdetermined, so our requirement is equivalent with saying that its matrix has maximum rank $dk$. As in the previous case, a generic choice of $\Gamma'$ will guarantee that the first $k$ equations above will give rise to a compatible system.

Let us now see why the function $H$ is injective. Assume that $(\vec{\T}_1,\ldots,\vec{\T}_d),(\vec{\V}_1,\ldots,\vec{\V}_d)\in\H$ have the same value under $H$. Again, by using the system \eqref{solvablesystemsecondversiontttsdute76rt346r43fjhgrgrg}, we first obtain that
$\vec{\T}_{i_0}=\vec{\V}_{i_0}$ for some $i_0$. There are two type of scenarios that will sustain the domino effect.

If $i_0$ happens to be even, its neighbor(s) $i$ will be odd, and thus we are guaranteed that $f_j(\vec{\T}_i)=f_j(\vec{\V}_i)$ for each $j\in\{3,\ldots,m+1\}$. However, since $i$ is a neighbor of $i_0$, it will also follow that $f_{j_0}(\vec{\T}_i)=f_{j_0}(\vec{\T}_{i_0})=f_{j_0}(\vec{\V}_{i_0})=f_{j_0}(\vec{\V}_{i})$, where $j_0$ is either 1 or 2, depending on whether $i=i_0-1$ or $i=i_0+1$. In any case, $\vec{\T}_{i}$ and $\vec{\V}_{i}$ will share $m$ standard projections, and thus will have to coincide, by
Proposition \ref{knowmknowall}.

The second scenario is when $i_0$ is odd. In this case we can not prove by following the same procedure that $\vec{\T}_{i}=\vec{\V}_{i}$, for a neighboring $i$. What we do instead is consider the two-scheme(s), one associated with indices $i_0+1,i_0+2$, the other one associated with indices  $i_0-1,i_0-2$. Each of  these two-schemes is of the form \eqref{solvablesystemsecondversionhjuy}. Indeed, $2m-2$ of the equations are going to come from \eqref{solvablesystemsecondversiontttsdute76rt346r43fjhgrgrg}, while the additional equation will be of the form
$$G_{1}(\vec{\xi}^{(i_{0}+1)})=w_{i_0+1,1}$$
for the first two-scheme and
$$G_{2}(\vec{\xi}^{(i_{0}-1)})=w_{i_0-1,2}$$
for the second two-scheme. In other words, in the case of the first scheme (with a similar situation for the second scheme) we know that  $\vec{\T}_{i_0+1}$ and $\vec{\V}_{i_0+1}$ share $m-1$ standard projections, $\vec{\T}_{i_0+2}$ and $\vec{\V}_{i_0+2}$ share  $m-1$ standard projections, and in addition, we recall  that $f_2(\vec{\T}_{i_0+1})=f_2(\vec{\T}_{i_0+2})$ and  $f_2(\vec{\V}_{i_0+1})=f_2(\vec{\V}_{i_0+2})$. The analysis of this two-scheme (essentially, a repeat of the argument from the case $d=2$) will imply that $\vec{\T}_{i}=\vec{\V}_{i}$, for $i\in\{i_0+1,i_0+2\}$.

If we allow  combinations of these scenarios, it is easy to see that the domino effect will eventually prove that $\vec{\T}_{i}=\vec{\V}_{i}$, for all $i$.

\textbf{Subcase 2a: $n$ is odd and $d$ is even}. This is the most delicate case. We will use the same construction as in the previous subcase, but with $d+1$ vector trees, rather than $d$. More precisely, without loss of generality we can assume that
$$
\max\{|\F_j(l_j,\vec{x})|:1\le j\le m+1\}=|\F_2(l_2,\vec{x})|\le\min\{|\F_j(l_j,\vec{x})|:m+2\le j\le n\}.
$$
Apply Lemma \ref{./x,.ewuf490389} (this time for $d+1$ sets) with $A_j=\F_1(l_1,\vec{x})$ and $g_j=f_1$ if $j$ is odd, $A_j=\F_2(l_2,\vec{x})$ and $g_j=f_2$ if $j$ is even. Here $j\in\{1,\ldots,d+1\}$. Also, define
$$\H:=\{(\vec{\T_1},\ldots,\vec{\T_{d+1}})\in X^{d+1}:g_j(\vec{\T}_{j})=g_j(\vec{\T}_{j+1}),1\le j\le d\},$$
$$H:\H\to \prod_{j}B_j,$$
where each $B_j$ equals one of the sets $\F_i(l_i,\vec{x})$, and
$$H(\vec{\T}_1,\ldots,\vec{\T}_{d+1}):=$$
$$(f_3(\vec{\T}_1),f_4(\vec{\T}_1),\ldots,f_{m}(\vec{\T}_1),f_{3}(\vec{\T}_2),f_{4}(\vec{\T}_2),\ldots,f_{m}(\vec{\T}_2),\ldots,
f_3(\vec{\T}_{d+1}),f_4(\vec{\T}_{d+1}),\ldots,f_{m}(\vec{\T}_{d+1}),$$
$$f_{m+1}(\vec{\T}_1),f_{m+1}(\vec{\T}_3),\ldots,f_{m+1}(\vec{\T}_{d+1})).$$
In addition to requiring that the two-schemes \eqref{solvablesystemsecondversionhjuy} have at most one solution, we will also need that the following system (a copy of \eqref{solvablesystemsecondversiontttsdute76rt346r43fjhgrgrg} with $d\mapsto d+1$)
\begin{equation}
\label{solvablesystemsecondversiontttsdute76rhhhhhhhhfffvvv}
\begin{cases} G_1(\vec{\xi}^{(1)})-G_1(\vec{\xi}^{(2)})=v_1
\\ G_2(\vec{\xi}^{(2)})-G_2(\vec{\xi}^{(3)})=v_2
\\ \dots\ldots\ldots
\\ G_{2}(\vec{\xi}^{(d-1)})-G_{2}(\vec{\xi}^{(d)})=v_{d-1}
\\ G_l(\vec{\xi}^{(j)})=v_{1,3},\;\;3\le l\le m,\;1\le j\le d+1
\\ G_{m+1}(\vec{\xi}^{(1)})=v_{1,m+1}\\
\\ G_{m+1}(\vec{\xi}^{(3)})=v_{3,m+1}\\
\ldots\ldots\ldots \\
\\ G_{m+1}(\vec{\xi}^{(d+1)})=v_{d+1,m+1}
\end{cases}
\end{equation}
has at most one solution. There are $(d+1)(m-\frac{1}{2})-\frac12$ $\R^d$ valued equations in $d+1$ variables from $\R^k$. The fact that the system is overdetermined
$$[(d+1)(m-\frac{1}{2})-\frac12]d\ge (d+1)k$$
is a consequence of the fact that $\frac{k}{d}<m-\frac12$.

The argument will then run as in the previous subcase. We leave details to the interested reader.

\subsection{The case $d=1$}
\label{$d=1$}

As advertised earlier, a simplified version of the combinatorial argument in Section \ref{The case $d=2$} can also handle the case $d=1$, and more generally, the case of arbitrary $d$ and $m<n/2$. We present the argument in this generality, and thus assume $m<n/2$, rather than $d=2$. We note again that the argument here  reproves the main Theorem in \cite{MTT1}, in the locally $L^2$ case, without any appeal to induction.

 We perform the same two stage decomposition, exactly as in Section \ref{The case $d=2$}. This time however the function $H$ will have a simpler form. More precisely, in the first stage one considers functions of the form
$$H:\F^{**}(\vec{l},\vec{x})\to \prod_{j=1}^m\F_j(l_j,\vec{x})$$ given by
$$H(\vec{\T}):=(f_1(\vec{\T}),\ldots,f_m(\vec{\T})),$$
with an identical construction (up to notation) for the second stage. The injectivity of $H$ will follow from Proposition \ref{knowmknowall} in the first stage, and from (r2) in the second stage.

 The injectivity of $H$ will in turn imply the desired estimates \eqref{estforcountfunction1fgfdgfh} and \eqref{chvhcggyeuiwqyduy3e8394945056=-067asxwqr3466uuk}, since $m<n/2$.

 Note that Lemma \ref{lemakatztao} and Proposition \ref{p:skhxghduyetd78436r2843908} are no longer needed here, and as a consequence we do not need any further non-degeneracy conditions on $\Gamma'$, other than the one from  \cite{MTT1}, namely that $\Gamma'$ is parametrizable over any $k$ canonical variables.

\section{Proof of Theorem \ref{redu2}}

Assume $d=2$. The argument for $d>2$ would follow with no essential modification.

For each $\P\subset\S$ we will use the notation
$$\Lambda_{\P}(\vec{x})(F_1,\ldots,F_n):=\sum_{s\in\P}|R_s|^{-\frac{n}{2}}\prod_{i=1}^{n}|\langle F_i,\phi_{s_i}\rangle|1_{R_s}(\vec{x}).$$
By invoking interpolation and the dilation invariance of our operator, it suffices to prove that for each $2<p_i\le \infty$ with $\frac{1}{p_1}+\ldots+\frac{1}{p_n}=\frac{1}{p}$ and each $\|F_i\|_{p_i}=1$ we have
\begin{equation}
\label{hwsge6eydgtgxvxastxwqtlslw'qo0-o5495lp}
|\{\vec{x}:\Lambda_{\S}(F_1,\ldots,F_n)(\vec{x})\gtrsim 1\}|\lesssim 1.
\end{equation}
For the remaining part of the argument, the functions $F_i$ will be fixed as above, and all sizes will be computed with respect to them.

Consider the exceptional set
$$E=\bigcup_{i=1}^{n}\{\vec{x}:M_2F_i(\vec{x})\ge 1\},$$
and note that $|E|\lesssim 1.$

It now suffices to restrict attention in \eqref{hwsge6eydgtgxvxastxwqtlslw'qo0-o5495lp} to the collection (which for simplicity will also be denoted with $\S$) of multi-tiles $s$ which have the property that
$R_s\cap E^c\not=\emptyset$. Lemma \ref{sizeestbymaxf} will now imply that
$\size_i(\S_i)\lesssim 1$.

Apply now Proposition \ref{p:Bes2} to the collections $\S_i$. We get that
$$\S_i:=\bigcup_{2^{-k}\lesssim 1} \S^{(k)}_i,$$
where
\begin{equation}
\label{129eioavxzsdi34[rpkjkaDAR}
\size_i(\S^{(k)}_i)\le 2^{-k+1}
\end{equation}
and each $\S^{(k)}_i$ is the (disjoint) union of a family $\F^{(k)}_i$ of trees such that
\begin{equation}
\label{aza5ytqos90e-930i0lkjk'[}
\|N_{\F^{(k)}_i}\|_{q}\lesssim 2^{2k}2^{kp_i/q},\;1\le q<\infty.
\end{equation}

An immediate consequence of \eqref{aza5ytqos90e-930i0lkjk'[} (choose $q$ large enough) is that
$$|\{\vec{x}:N_{\F^{(k)}_i}(\vec{x})>2^{4k}\}|\lesssim 2^{-10k}.$$
By eliminating another exceptional set of measure $O(1)$, it thus suffices to further restrict attention in \eqref{hwsge6eydgtgxvxastxwqtlslw'qo0-o5495lp} to those $\vec{x}$ which satisfy
\begin{equation}
\label{jhjhwtq689y1uwevectortreeshihihihih}
N_{\F^{(k)}_i}(\vec{x})\le 2^{4k}
\end{equation}
for each $i$ and each $2^{-k}\lesssim 1$.

We fix some $k_i$ for each $i$, and denote by $\S(\vec{k})$ the collection of all multi-tiles $s$ with $s_i\in \S_i^{(k_i)}$ for each $i$. We follow the procedure described in Section \ref{The case $d=2$}, applied to $\P:=\S(\vec{k})$ and $\F_i:=\F^{(k_i)}_i$. The collection $\S(\vec{k})$ will be the union of three families of vector trees: $\F^{**}_{\vec{k}}$ from the first stage of the construction and
\begin{equation}
\label{vectortreeshihihihih}
\F^{*}_{\vec{k}}:=\bigcup_{\vec{l}:1\le 2^{l_i}\le 2^{4k_i}}\bigcup_{\vec{r}:2^{-r_i}\le 2^{-k_i+1}}\F_{\vec{k}}(\vec{r},\vec{l}),
\end{equation}
$$\F_{\vec{k}}:=\bigcup_{\vec{l}:2^{l_i}>2^{4k_i}}\bigcup_{\vec{r}:2^{-r_i}\le 2^{-k_i+1}}\F_{\vec{k}}(\vec{r},\vec{l})$$
from the second stage of the construction. Due  \eqref{wdwer4t6hestforcountfunction1} and \eqref{jhjhwtq689y1uwevectortreeshihihihih}, the family $\F_{\vec{k}}$ can be ignored, since it will not contribute to $\Lambda_{\S(\vec{k})}$.

Let $t$ be a sufficiently large number. We plan to evaluate $\|N_{\F^{**}_{\vec{k}}}\|_t$, and in doing so we will invoke \eqref{estforcountfunction1fgfdgfh}, \eqref{aza5ytqos90e-930i0lkjk'[} and H\"older's inequality:
$$\|N_{\F^{**}_{\vec{k}}}\|_t\lesssim \prod_{i=1}^{n}2^{k_i(2\delta+\frac{p_i}{nt})}.$$
Combining this with Lemma \ref{vecttreeest} and estimate \eqref{129eioavxzsdi34[rpkjkaDAR} for the size, we get
$$\|\Lambda_{\F^{**}_{\vec{k}}}\|_t\lesssim \prod_{i=1}^{n}2^{k_i(2\delta+\frac{p_i}{nt}-1)}.$$
Since $\delta<1/2$, it follows that
$$|\{\vec{x}:\Lambda_{\bigcup_{\vec{k}:2^{-k_i}\lesssim 1}\F^{**}_{\vec{k}}}(F_1,\ldots,F_n)(\vec{x})\gtrsim 1\}|\lesssim 1.$$
It now remains to evaluate the contribution coming from the vector trees in \eqref{vectortreeshihihihih}. Fix $\vec{r},\vec{l}$. As before, by using \eqref{anotherBesselinLq}, \eqref{chvhcggyeuiwqyduy3e8394945056=-067asxwqr3466uuk} and H\"older's inequality we get
$$\|N_{\F_{\vec{k}}(\vec{r},\vec{l})}\|_t\lesssim \prod_{i=1}^{n}2^{r_i(2\delta+\frac{p_i}{nt})}.$$
Combining this with Lemma \ref{vecttreeest} and estimate \eqref{defertfedderdedrdsllkjj} for the size we get
$$\|\Lambda_{\F_{\vec{k}}(\vec{r},\vec{l})}\|_t\lesssim \prod_{i=1}^{n}2^{r_i(2\delta+\frac{p_i}{nt}-1)}.$$
Summing this first over $2^{-r_i}\le 2^{-k_i+1}$, then over $0\le l_i\le 4k_i$ and finally over $2^{-k_i}\lesssim 1$ we get
$$|\{\vec{x}:\Lambda_{\bigcup_{\vec{k}}\F^{*}_{\vec{k}}}(F_1,\ldots,F_n)(\vec{x})\gtrsim 1\}|\lesssim 1.$$
This finishes the argument.


\begin{thebibliography}{99}
\bibitem{DT} Demeter C., Thiele C., {\em On the two dimensional Bilinear Hilbert Transform}, to appear in Amer. J. of Math.
\bibitem{Fe} Fefferman C. {\em Pointwise convergence of Fourier series},  Ann. of Math. (2)  \textbf{98}  (1973), 551-571.
\bibitem{GW} W.T. Gowers, J. Wolf {\em The true complexity of a system of linear equations}, available at http://arxiv.org/abs/0711.0185
\bibitem{GL1} Grafakos L. and Li X., {\em Uniform bounds for the bilinear  Hilbert transform I },
Ann. of Math. {\bf 159.3}, pp. 889-993, [2004].
\bibitem{KT} N. Katz, T. Tao, {\em  Bounds on arithmetic projections, and applications to the Kakeya conjecture} Math. Res. Lett. 6  (1999),  no. 5-6, 625-630.
\bibitem{LT1} Lacey M. and Thiele C., {\em $L^p$ bounds on the bilinear Hilbert transform for $2<p<\infty $}, Ann. of Math. {\bf 146}, pp. 693-724, [1997].
\bibitem{LT2} Lacey M. and Thiele C., {\em On Calder\'on's conjecture.},
Ann. of Math. {\bf 149.2}, pp. 475-496, [1999].
\bibitem{MTT1} Muscalu, C., Tao, T. and Thiele, C. {\em Multilinear
operators given by singular multipliers} J. Amer. Math. Soc. \textbf{15} (2002),no. 2, 469-496.
\bibitem{MTT3} Muscalu, C., Tao, T. and Thiele, {\em $L\sp p$ estimates for the biest. II. The Fourier case.}  Math. Ann. {\bf 329}  (2004),  no. 3, 427-461.
\end{thebibliography}
\end{document}